\documentclass[11pt,oneside]{amsart}   	
\usepackage{geometry}                		
\usepackage[parfill]{parskip}    		
\usepackage{graphicx,tikz}				
\usepackage{array, makecell} 
\usepackage{amsmath, amsrefs, amssymb, amsthm}
\usepackage[colorlinks=true,linkcolor=magenta,citecolor=cyan,urlcolor=magenta,linktocpage=true]{hyperref}
\usepackage{cleveref}
\usepackage{mathtools}
\usepackage{enumerate,enumitem}
\usepackage{shuffle}
\usepackage[most]{tcolorbox}
\usepackage{caption}
\usepackage{genyoungtabtikz}
\YFrench
\usepackage[all,cmtip]{xy}

\newcommand\ylw{\Yfillcolour{yellow}}
\newcommand\ygreen{\Yfillcolour{green!80!black}}
\newcommand\ywhite{\Yfillcolour{white}}
  \Yboxdim{15pt}
\Yfillcolour{green!80!black}

\newcommand\rouge{\textcolor{red}}
\newcommand\bleu{\textcolor{blue}}
\newcommand\vertfonce{\textcolor{green!70!black}}

\theoremstyle{plain}
\newtheorem{theorem}{Theorem}[section]
\newtheorem{proposition}[theorem]{Proposition}
\newtheorem{conjecture}[theorem]{Conjecture}
\newtheorem{lemma}[theorem]{Lemma}
\newtheorem{corollary}[theorem]{Corollary}

\theoremstyle{definition}
\newtheorem{definition}[theorem]{Definition}
\newtheorem{problem}[theorem]{Problem}

\newtheorem{example}[theorem]{Example}
\newtheorem{remark}[theorem]{Remark}

\numberwithin{equation}{section}

\makeatletter
\renewcommand*{\eqref}[1]{%
 \hyperref[{#1}]{\textup{\tagform@{\ref*{#1}}}}%
}
\makeatother

\Crefname{definition}{\textcolor{magenta}{Definition}}{\textcolor{magenta}{Definitions}}

\newcommand{\Z}{\boldsymbol{z}}
\newcommand{\Y}{\boldsymbol{y}}
\newcommand{\YY}{\mathcal{Y}}
\newcommand{\XX}{\mathcal{X}}
\newcommand{\X}{\boldsymbol{x}}

\newcommand{\N}{\mathbb{N}}
\renewcommand{\S}{\mathbb{S}}
\newcommand{\Circle}{\raisebox{-2.5pt}{\Huge{$\circ$}}}

\DeclareMathOperator{\DHn}{DH}

\DeclareMathOperator{\area}{area}
\DeclareMathOperator{\dinv}{dinv}

\DeclareMathOperator{\sign}{sign}

\DeclareMathOperator{\maj}{maj}
\DeclareMathOperator{\pmaj}{pmaj}
\DeclareMathOperator{\comaj}{comaj}
\DeclareMathOperator{\revmaj}{revmaj}
\DeclareMathOperator{\revcomaj}{revcomaj}

\DeclareMathOperator{\rw}{rw}

\DeclareMathOperator{\Des}{Des}
\DeclareMathOperator{\asc}{asc}
\DeclareMathOperator{\Asc}{Asc}
\DeclareMathOperator{\D}{D}

\DeclareMathOperator{\weight}{weight}
\DeclareMathOperator{\Ht}{\widetilde{H}}

\DeclareMathOperator{\PR}{PR}

\DeclareMathOperator{\R}{R}
\DeclareMathOperator{\CC}{CC}

\DeclareMathOperator{\ith}{th}
\DeclareMathOperator{\st}{st}
\DeclareMathOperator{\bC}{{\bf CC}}
\DeclareMathOperator{\LC}{{LC}}

\DeclareMathOperator{\aPF}{{aPF}}

\DeclareMathOperator{\WV}{{WV}}
\DeclareMathOperator{\LW}{{LW}}
\DeclareMathOperator{\spl}{{split}}
\DeclareMathOperator{\join}{{join}}

\DeclareMathOperator{\PP}{\mathcal{P}}
\DeclareMathOperator{\LD}{LD}
\DeclareMathOperator{\LLD}{LLD}

\DeclareMathOperator{\RLD}{{RLD}}
\DeclareMathOperator{\RLLD}{{RLLD}}

\DeclareMathOperator{\SR}{{SR}}

\DeclareMathOperator{\bSM}{{SM}}
\DeclareMathOperator{\ret}{{ret}}

\DeclareMathOperator{\RP}{{RP}}

\DeclareMathOperator{\SP}{{SP}}
\DeclareMathOperator{\SC}{{SC}}
\DeclareMathOperator{\LSP}{{LSP}}
\DeclareMathOperator{\MLD}{{MLD}}

\DeclareMathOperator{\wt}{wt}
\DeclareMathOperator{\lwt}{lwt}
\DeclareMathOperator{\col}{col}

\let\oldXi\Xi
\let\Xi\undefined
\DeclareMathOperator{\Xi}{\oldXi}

\title{The Super Nabla Operator}

\date{\today} 		

\author{Fran\c{c}ois Bergeron}
\address{Universit\'{e} du Qu\'{e}bec \`{a} Montr\'{e}al \\ LACIM}
\email{bergeron.francois@uqam.ca}

\author{Jim Haglund}
\address{University of Pennsylvania \\ Department of Mathematics}
\email{jhaglund@upenn.edu}

\author{Alessandro Iraci}
\address{Universit\`{a} di Pisa \\ Dipartimento di Matematica}
\email{alessandro.iraci@unipi.it}

\author{Marino Romero}
\address{Universit\"{a}t Wien \\ Fakult\"{a}t f\"{u}r Mathematik}
\email{marino.romero@univie.ac.at}

\begin{document}

\begin{abstract}
	We consider here a new operator, called ``super nabla'', which is shown to be generic among operators for which the modified Macdonald polynomials are joint eigenfunctions. All previously known Macdonald eigenoperators can readily be obtained from super nabla, including the usual nabla operator, the Delta operators, and other operators that have appeared in the literature. Thus, the super nabla operator furnishes an overall unified viewpoint on this family of operators, as well as opening up new possibilities. We prove several new identities arising from specializations of the parameters $q$ and $t$ involved in the specification of these operators, as well as unifying combinatorial interpretations.
\end{abstract}

\maketitle
{\setcounter{tocdepth}{1}\parskip=0pt\footnotesize \tableofcontents}



\section{Introduction}

Recent years have seen an ever growing body of research on algebras constructed out of operators for which Macdonald polynomials form a joint basis of eigenfunctions. Related subjects range from Algebraic Combinatorics to Statistical Mechanics, going through Representation Theory, Algebraic Geometry, as well as Link Homology, to name but a few. We consider here the modified version  of Macdonald polynomials, denoted $\widetilde{H}_\mu(\X;q,t)$, for $\mu$ ranging over all integer partitions. Among many nice properties, they are known to have coefficients in $\N[q,t]$ when expanded in the basis of Schur functions (see~\cite{Haiman-nfactorial-2001}). For instance,
\begin{align*}
	& \widetilde{H}_{3}(\X;q,t) = s_3(\X) + (q+q^2) s_{21}(\X) + q^3 s_{111}(\X),   \\
	& \widetilde{H}_{21}(\X;q,t) = s_3(\X) + (q+t) s_{21}(\X) + qt\, s_{111}(\X),   \\
	& \widetilde{H}_{111}(\X;q,t) = s_3(\X) + (t+t^2) s_{21}(\X) + t^3 s_{111}(\X).
\end{align*}
It is usual to express this fact by saying that they are \emph{Schur positive}\footnote{With coefficients that are $(q,t)$-polynomials having positive integer coefficients.}. This Schur positivity has been naturally explained by showing (see~\cite{Garsia-Haiman-qLagrange-1996}) that $\widetilde{H}_\mu(\X;q,t)$ is the Frobenius characteristic of a bigraded $\S_n$-module $\mathcal{H}_\mu$, with the parameters $q$ and $t$ serving to encode the grading, and the Schur functions encoding $\S_n$-irreducibles. These \emph{Garsia-Haiman modules} $\mathcal{H}_\mu$ appear as sub-modules of the space of \emph{Diagonal Harmonics},
\[ \DHn_n = \left\{P \in \mathbb{C}[u_1,v_1,\dots, u_n,v_n] \mid \sum_{i=1}^n \partial_{u_i}^r \partial_{v_i}^s P = 0, \text{ whenever $r+s>0$} \right\}, \]
under the $\S_n$-\emph{action} that \emph{diagonally} permutes both sets of variables, i.e.\ $\sigma \cdot u_i = u_{\sigma(i)}$ and $\sigma \cdot v_i = v_{\sigma(i)}$.

Recall that the Frobenius characteristic of a bigraded $\S_n$-module $\mathcal{W} = \bigoplus_{(r,s)} \mathcal{W}^{(r,s)}$
is the symmetric function (in the variables $\X=(x_1,x_2,x_3,\ldots)$)
\[ \mathcal{F}(\mathcal{W})(\X;q,t) = \sum_{(r,s)\in \N\times\N} q^r t^s \sum_{\lambda \vdash n} n^{(r,s)}_\lambda s_\lambda(\X), \]
assuming that each component $\mathcal{W}^{(r,s)}$ decomposes as an $\S_n$-module as
\[ \mathcal{W}^{(r,s)} \simeq \bigoplus_{\lambda \vdash n } n^{(r,s)}_{\lambda} V^{\lambda}, \]
with $V^{\lambda}$ denoting (representatives of) Young's irreducible module indexed by $\lambda$, and  $n^{(r,s)}_\lambda$ denotes its multiplicity in $\mathcal{W}^{(r,s)}$. In other terms, the coefficient of a Schur function $s_\lambda(\X)$ in $\mathcal{F}(\mathcal{W})(\X;q,t)$ is the graded Hilbert series of occurrences of the irreducible $V^{\lambda}$ in $\mathcal{W}$.

One of the main conjectures in \cite{Garsia-Haiman-qLagrange-1996}, proven by Haiman in \cite{Haiman-nfactorial-2001}, states that
\begin{align}
	\label{nabla_en}
	\mathcal{F}\left( \DHn_n \right)(\X;q,t) = \nabla (e_n(\X)),
\end{align}
where $e_n(\X)$ stands for the degree $n$ \emph{elementary symmetric function}, and $\nabla$ is the Nabla operator of 
\cite{Bergeron-Garsia-ScienceFiction-1999} which has the (modified) Macdonald polynomials $\Ht_\mu (\X) = \Ht_\mu(\X;q,t)$ as eigenfunctions with eigenvalue
\begin{align*}
	\nabla  \Ht_\mu  (\X)= \Big(\prod_{(i,j)\in\mu} q^it^j\Big)  \Ht_\mu (\X).
\end{align*}
Here, $(i,j)$ runs over the set of \emph{cells} of $\mu$ (see \Cref{gammaparkingfunctions}). Among other similarly defined operators, we have
\begin{align}
	\label{Definition:Mac_Operators}
	\Delta_{e_1} \Ht_\mu (\X)= \Big(\sum_{(i,j)\in\mu} q^it^j\Big)  \Ht_\mu (\X), &  & {\rm and} &  &
	\Pi \Ht_\mu (\X)= \Big(\prod_{\substack{(i,j)\in\mu                                               \\ (i,j)\neq (0,0)}} (1-q^it^j)\Big)  \Ht_\mu (\X).
\end{align}

The main contribution of this work concerns the following extension of $\nabla$, which both unifies many previous results and opens new intriguing possibilities.

\begin{definition}
	\label{Definition:super-nabla}
	We define $\nabla_\ast \colon \Lambda \rightarrow \Lambda \otimes \Lambda$ as
	\begin{equation}
		\label{Equation:super-nabla}
		\nabla_{\ast} \Ht_\mu \coloneqq \Ht_\mu \otimes \Ht_\mu.
	\end{equation}
	Alternatively, if we consider a second set of variables $\Y$, we simply set
	\begin{equation}
		\nabla_{\Y} \Ht_\mu(\X) \coloneqq \Ht_\mu(\Y) \Ht_\mu(\X).
	\end{equation}
\end{definition}

From now on, we identify tensor products with multiplication of symmetric functions in different sets of variables.

For any symmetric function $F$, it is handy to extend the multi-variate Hall inner product to act on operators, as well as it does on symmetric functions. This way, we write $\langle F, \nabla_\ast \rangle$ for the operator that sends $H_\mu$ to
\begin{align}
	\langle F, \nabla_\ast \rangle \Ht_\mu \coloneqq \langle F, \Ht_\mu \rangle \Ht_\mu,
\end{align}
with $\langle F, \Ht_\mu \rangle$ standing for the usual version of the Hall scalar product. With this notation in hand, we have the following identities.

\begin{proposition}
	\label{prop:scalar_product}
	For any homogeneous symmetric function $F(\X)$ of degree $n$,
	\begin{align}
		\langle p_n,\nabla_\ast \rangle F = \Pi F &  & \text{and} &  & \langle e_k h_{n-k}, \nabla_\ast \rangle F = \Delta_{e_k} F.
	\end{align}
\end{proposition}

One of our main goals is to give combinatorial models for expressions of the form $\nabla_\ast^k F$ for interesting symmetric functions $F$. In particular, we consider the cases when\footnote{Recall that it is usual in the subject to denote by $M$ the product $(1-q)(1-t)$.} \[ F = \Xi e_\lambda \coloneqq \Delta_{e_1} M \Pi\, e_{\lambda}(\X/M) \] (as in \cite{IraciRomero2022DeltaTheta}), for which we give combinatorial expansions at $t=1$.

Note that the family of these $\Xi e_\lambda$ is a basis of $\Lambda$, and that $\Xi e_{\lambda,1} = \Delta_{e_1} \Theta_{e_\lambda} e_1$ in terms of Theta operators (see \cite{DAdderio-Iraci-VandenWyngaerd-Theta-2021}). Indeed, we will show that it is enough to understand the case $\lambda = (n)$ (in which case $\Xi e_n = e_n)$, meaning that the symmetric function $\nabla_\ast^{k+1} e_n$ contains all the information about the effect of the operator $\nabla_\ast^k$ on any symmetric function.

It will be convenient to set $\XX_k = \X^{\otimes k}$ (that is, $\X \otimes \X \otimes \dots \otimes \X$, $k$ times), where $k$ can be omitted if it is clear from the context, and, for any appropriate tuple of tuples $w = (w_1, \dots, w_n)$, with $w_i = (w_{i1}, \dots, w_{ik})$, set \[ \XX^w = (x_{w_{11}} x_{w_{21}} \cdots x_{w_{n1}}) \otimes \dots \otimes (x_{w_{1k}} x_{w_{2k}} \cdots x_{w_{nk}}). \]

In this work, we extensively explore the special case when $t$ is specialized at $1$, giving results such as the following.
\begin{theorem}
	When $t=1$, we have the following monomial, elementary, and Schur function expansions, respectively:
	\begin{align*}
		\left. \nabla_\ast^k e_n \right \rvert_{t=1}
			& = \sum_{L \in \LD_{k^n}} q^{\area(L)} \XX_{k+1}^L                      \\
			& = \sum_{L \in \LD_{k^n}^k} q^{\area(L)} e_{\eta(L)} \otimes \XX_{k}^L \\
			& = \sum_{L \in \LLD^k_{k^n}} q^{\area(L)} e_{\eta(L)} \otimes s_{\lambda^1(w(L))} \otimes \cdots \otimes s_{\lambda^k (w(L))},
	\end{align*}
	where the first two sums are over multi-labeled Dyck paths in the $nk \times n$ rectangle, with different labelings and associated monomials, and the last sum is over lattice multi-labeled Dyck paths. See \Cref{section:multi-labeledDyckpaths} for details.
\end{theorem}

In \Cref{section:symmetricfunctions} we describe two approaches giving these expansions and elaborate on the general combinatorial case, from where we derive alternate expansions. This involves statistics that \emph{look right} (see \Cref{looksright}), which are also useful for the \emph{square case} expansions of  \Cref{section:power}.

This so called square case gives an analogous result when  $F = \Delta_{m_\gamma} (-1)^{n-1} p_n$ (see \Cref{section:symmetricfunctions} for the necessary definitions). In particular if $\gamma=1^n$ then $F = \nabla \omega(p_n)$, and we get
\begin{theorem}
	\begin{align*}
		\left. \nabla_\ast^k \nabla \omega (p_n) \right\rvert_{t=1}
			& = \sum_{L \in \RLD_{k^n}} q^{\area(L)} \XX_{k+1}^L \\
			& = \sum_{L \in \RLD^{1^n}_{k^n}} q^{\area(L)} \XX_{k}^L \otimes e_{\eta(L)} \\
			& = \sum_{\substack{P \in \RLLD^{1^n}_{k^n} \\ w= w(p)}} q^{\area(P)} s_{\lambda^1(w)} \otimes \cdots \otimes s_{\lambda^k (w)} \otimes e_{\eta(P)},
	\end{align*}
	where the first two sums are over multi-labeled $k^n$-Dyck paths with a marked return, and the third sum is over those multi-labeled paths whose labels give lattice words. See \Cref{section:power} for details.
\end{theorem}

\section{Combinatorial definitions}
\label{gammaparkingfunctions}

As usual we may represent partitions by their Ferrers diagram (using the French convention), so that the number of \emph{cells} in row $k$ is the size of the $k^{\ith}$ part of $\mu$.
Cells have usual cartesian coordinates, hence the cells of the $k^{\rm th}$ row have coordinates $(i,k-1)$, with $0\leq i< \mu_k$; and we write $c\in\mu$ when $c$ is a cell (of the diagram) of $\mu$.
For a partition $\mu$ and a cell $c \in \mu$, we let $a(c)$ and $\ell(c)$
respectively denote the {\sl arm} and {\sl  leg} of the cell $c$. These are, respectively, the number of cells in $\mu$ that lie strictly to the right of $c$ and the number of cells strictly above $c$. See \Cref{fig:(1)-armslegs} for an example.

\begin{figure}[ht]
	\begin{center}
		\includegraphics[scale=.5]{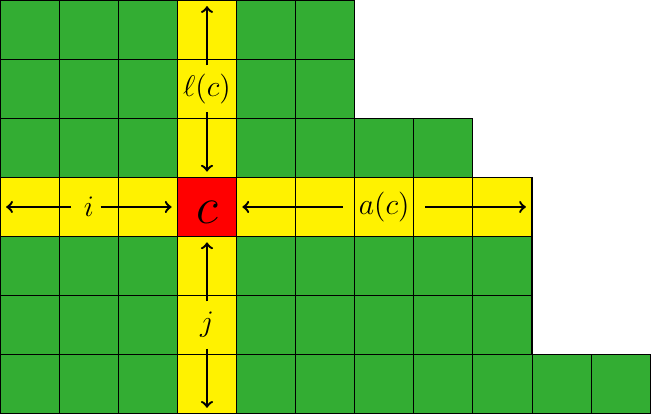}
	\end{center}
	\caption{The arm $a(c)$ and leg $\ell(c)$ of a cell $c=(i,j)$.}
	\label{fig:(1)-armslegs}
\end{figure}
We recall that we have the following notations in this context:
\begin{align}
	B_\mu(q,t)\coloneqq\sum_{(i,j)\in \mu} q^it^j, &                                                   & T_\mu(q,t)\coloneqq\prod_{(i,j)\in \mu} q^it^j, &
	                                               & \Pi_\mu(q,t)\coloneqq\prod_{\substack{(i,j)\in\mu                                                     \\ (i,j)\neq (0,0)}} (1-q^it^j),
\end{align}
and as is usual, $M$ stands for the product $(1-q)(1-t)$.

To avoid confusion, we make the following distinction between \emph{words} and \emph{tuples}. We use the term {\emph{word}} for a sequence $w = (w_1, \dots, w_n)$ whose \emph{letters} $w_i = (w_{i,1}, \dots, w_{i,r_i}) \in \mathbb{N}^{r_i}$ are \emph{tuples} of nonnegative integers. Hence $w_{i,j} = w_{ij}$ is the $j^{\text{th}}$ digit of the letter $w_i$.
For any word or tuple $\alpha$, we denote the number of its letters by $\ell(\alpha)$. Hence, $\ell(w_1,\dots, w_n)  =n$ and $\ell(w_i) = r_i$.
The size of a tuple $\alpha = (\alpha_1,\dots, \alpha_r)$ is given by $|\alpha| = \alpha_1+ \cdots + \alpha_r$. If each $\alpha_i$ is nonzero, then $\alpha$ is said to be a \emph{composition} and we write $\alpha \vDash |\alpha|$. If $\alpha$ is a composition whose entries are weakly decreasing, then $\alpha$ is a \emph{partition}, denoted by $\alpha \vdash |\alpha|$.  To illustrate, the word
\begin{align*}
	w=(2113,1133,3131)
\end{align*}
has length $\ell(w) =3$. The first letter of $w$ is the composition $w_1 = 2113$ whose length is $\ell(w_1) = 4$ and whose size is $|w_1| = 7$.

Recall that the \emph{descent set} and  \emph{ascent set} of  an $r$-tuple of integers  $\alpha = \alpha_1\cdots \alpha_r$ are respectively defined as
\begin{align*}
	\Des(\alpha) \coloneqq \{ 1 \leq i < r \mid \alpha_i > \alpha_{i+1} \} &  & {\rm and} &  &
	\Asc(\alpha) \coloneqq \{ 1 \leq i < r \mid \alpha_i < \alpha_{i+1} \}.
\end{align*}
Related to these are the following statistics:
\begin{align*}
	\maj(\alpha)    & \coloneqq \sum_{i \in \Des(\alpha)} i,     &  & \comaj(\alpha)  \coloneqq \sum_{i \in \Des(\alpha)} (n-i), \\
	\revmaj(\alpha) & \coloneqq \sum_{i \in \Asc(\alpha)} (n-i), &  & \revcomaj(\alpha)  \coloneqq \sum_{i \in \Asc(\alpha)} i.
\end{align*}
Note that $\revmaj$ and $\revcomaj$ actually are the $\maj$ and $\comaj$ of the reverse of the tuple, hence the name.
As is classical, to every composition $\alpha = (\alpha_1,\dots, \alpha_\ell)$ of $n$, we may associate the set of partial sums:
\begin{align*}
	\Sigma(\alpha) \coloneqq \{\Sigma_1(\alpha) , \dots, \Sigma_{\ell-1}(\alpha)\} &  & \text{ where } &  & \Sigma_r(\alpha) \coloneqq \alpha_1+\cdots+ \alpha_r,
\end{align*}
which establishes a bijection between compositions and subsets of $\{1,\cdots,n-1\}$.

For a composition $\alpha\vDash n$ and tuple $\beta = (\beta_1,\dots, \beta_n) \in \mathbb{N}^n$, let
\[ \revmaj_\alpha( \beta) \coloneqq \sum_r \revmaj(\beta_{\Sigma_r(\alpha)+1},\dots, \beta_{\Sigma_{r+1}(\alpha)}). \]

Given a word of tuples $w=(w_1,\dots, w_n)$, we say $i$ is an $r$-descent if there are precisely $r$ indices $j$ such that $w_{i,j}>w_{i,j+1}$. Writing $w_{\star r} \coloneqq (w_{1r}, \dots, w_{nr}) \in \mathbb{N}^n$ for the word formed by the $r^{\text{th}}$ digit of each tuple, we may set
\[ \revmaj_{\alpha}(w) \coloneqq \sum_{r} \revmaj_\alpha(w_{\star r}). \]
As an example, we have
\begin{align*}
	\revmaj_{(2,1)}( 211,113,313) & = \revmaj_{(2,1)}(213) + \revmaj_{(2,1)}(111) +\revmaj_{(2,1)}(133) \\
	                              & = 0+0+1.
\end{align*}

Let $m_i(\alpha)$ be the number of indices $j$ such that $\alpha_j = i$; that is, $m_i(\alpha)$ is the multiplicity of $i$ in $\alpha$. We denote the multiplicity type of $\alpha$ as $m(\alpha) \coloneqq 0^{m_0(\alpha)}1^{m_1(\alpha)} 2^{m_2(\alpha)} \cdots$.

A \emph{lattice word} is a sequence $\alpha = (\alpha_1, \dots, \alpha_r) \in \mathbb{N}_+^r$ such that, for all $1 \leq i, j \leq r$, \[ m_{j+1}(\alpha_1,\dots,\alpha_i) \leq m_j(\alpha_{1},\dots, \alpha_{i}). \]
This means that $\alpha$ is a tuple such that every prefix has at least as many $1$'s as $2$'s, at least as many $2$'s as $3$'s, and so on. Such a word encodes a standard tableau where $w_i$ gives the row in which $i$ is placed. If the shape of the resulting tableau is $\lambda$, then we say that $w$ is a lattice word of type $\lambda = (m_1(\alpha), m_2(\alpha),\dots)$ and write $\lambda(w) = \lambda$. For instance $1121321$ is a lattice word of type $\lambda(1121321) = (4,2,1)$.

Denote by $R(\beta)$ the set of all tuples $\alpha = (\alpha_1, \dots, \alpha_r)$ whose entries can be rearranged to give $\beta$, so that $m(\alpha) = m(\beta)$.
It is convenient to collect the multiplicities of a sequence of tuples $w= (w_1,\dots, w_r)$, with $w_i \in \mathbb{N}^{k}$, into a single word of multiplicities
$m(w) = ( m(w_{\star 1}) , \dots, m(w_{\star k}) )$,  which we call the \emph{multiplicity type} of $w$.

We define the set of \emph{partition vectors of size $\beta$ rearranging to $\alpha$} as
\begin{align*}
	\PR(\alpha,\beta) & \coloneqq \{ \vec{w} = (w^1,\dots, w^{\ell(\vec{w})}) \mid w^i \vdash \beta_i ~ \text{ and } ~ w^1\cdots w^{\ell(w)} \in R(\alpha) \}.
\end{align*}
This is the set of sequences of partitions with respective sizes given by the entries of $\beta$, and whose collective union of parts rearranges to $\alpha$.

\section{Symmetric function manipulations}
\label{section:symmetricfunctions}
Our goal in this section is to recall some symmetric function manipulations needed to derive upcoming expansions.
For a given symmetric function $F(\X)$, we will write \[ \widehat{F}(\X) = F \left[ \frac{\X}{1-q} \right] \]
so that the principal specialization $F(1,q,q^2,\dots)$ may be written as $\widehat{F}(1)  = F \left[ 1/(1-q) \right]$.

\subsection{Combinatorics of forgotten symmetric functions}

For $\mu \vdash n$ of length $\ell$, the combinatorial formula for the forgotten symmetric function $f_\mu$ \cite{Egecioglu-Remmel-Bricks} is given by
\[ f_\mu(\X)= (-1)^{n-\ell} \sum_{ \alpha \in \R(\mu)} \sum_{i_1 \leq \cdots \leq i_{\ell}} x_{i_1}^{\alpha_1} \cdots x_{i_\ell}^{\alpha_\ell}. \]
Plethystically substituting $ (1-q)^{-1}$ for $\X$, we get the expansion
\begin{equation}
	\label{eq:fmu}
	\widehat{f}_\mu(1) = (-1)^{n-\ell} \sum_{ \alpha \in \R(\mu)} \sum_{0 \leq i_1 \leq \cdots \leq i_{\ell}} \left( q^{i_1} \right)^{\alpha_1} \cdots \left( q^{i_\ell} \right)^{\alpha_\ell}.
\end{equation}

\begin{definition}
	\label{def:CCT}
	For a given partition $\mu$ of  $n$, a \emph{column-composition tableau} of type $\mu$ is a pair $C = (\alpha, c)$ where $\alpha \in \R(\mu)$ is a composition that rearranges to $\mu$, and $c = (c_1 \leq c_2 \leq \dots \leq c_n)$ is a sequence such that \[ c_i < c_{i+1} \implies i \in \Sigma(\alpha). \]
	We denote by $\CC_\mu$ the set of column-composition tableaux of type $\mu$, and by $\overline{\CC}_\mu$ the subset of those such that $c_1 = 0$. For $C \in \CC_\mu$, we define the \emph{length} of $C$ as $\ell(C) = \lvert \mu \rvert$ and \emph{size} of $C$ as $\lvert C \rvert = c_1 + c_2 + \dots + c_n$. We will write $c_i(C)$ for $c_i$ when we need to specify the size of column $i$ in the column-composition tableau $C$.
\end{definition}

\begin{example}
	We can depict the elements of $\CC_\mu$ as follows.
	\begin{center}
		\begin{tikzpicture}[xscale=-1,scale=.5]
			\Yboxdim{1cm}
			\tyng(0cm,0cm,9,7,7,4,4)
			\ylw
			\tyng(0cm,-1cm,11)
			\begin{large}
				\draw[-,line width=.4mm] (-1,0)-- (12,0);
				\draw[-,line width=.4mm] (1,-1.5)-- (1,5.5);
				\draw[-,line width=.4mm] (4,-1.5)-- (4,5.5);
				\draw[-,line width=.4mm] (6,-1.5)-- (6,3.5);
				\draw[-,line width=.4mm] (7,-1.5)-- (7,3.5);
				\draw[-,line width=.4mm] (9,-1.5)-- (9,2.5);
			\end{large}
		\end{tikzpicture}
	\end{center}
	Here, we have $C = (\alpha, c)$ with $\alpha = 221231$ and $c = (0,0,1,1,3,3,3,5,5,5,5)$. Indeed $\mu = 322211$ (so $\alpha \in \R(\mu)$), $\lvert C \rvert = \sum_i c_i = 31$ is the number of green cells, and $\ell(W) = \lvert \mu \rvert = 11$. Since $c_1(C) = 0$ (there are no green cells in the first column), we have $C \in \overline{\CC}_{(3,2^3,1^2)}$.
\end{example}

\begin{definition}
	\label{def:CCmu}
	Denote the generating function of column-composition tableaux of type $\mu$ by
	\begin{align}
		\bC_\mu(q)            & \coloneqq \sum_{C \in \CC_\mu} q^{|C|},~ \text{ and }                             \\
		\overline{\bC}_\mu(q) & \coloneqq (1-q^{|\mu|})  {\bC}_\mu(q) =  \sum_{C \in \overline{\CC}_\mu} q^{|C|}.
	\end{align}
\end{definition}

By construction, these series give the principal specialization of the forgotten basis (see \cite{IraciRomero2022DeltaTheta}*{Proposition~5.2}).

\begin{proposition} For any partition $\mu$,
	\begin{align}
		 & \widehat{f}_\mu(1) = (-1)^{|\mu|-\ell(\mu)} \bC_\mu(q), ~ \text{ and }           \\
		 & (1-q^{|\mu|}) \widehat{f}_\mu(1) = (-1)^{|\mu|-\ell(\mu)} \overline{\bC}_\mu(q).
	\end{align}
\end{proposition}
We also have the following specialization. The set of column-composition tableaux of type $(1^n)$ are in bijection with partitions whose largest part is at most $n$. Indeed, as $f_{1^n} = h_n$, the bijection is explained by the following identity of generating series.
\begin{corollary}
	\label{cor:forgotten}
	For any $n$, \[ \widehat{f}_{1^n}(1) = \widehat{h}_n(1) = \frac{1}{(1-q)\cdots(1-q^n)}. \]
\end{corollary}

\subsection{Classical properties of the Macdonald polynomials}
In plethystic notation, the classical Cauchy (kernel) identities state that
\begin{proposition}
	\label{prop:Cauchy_formula} For any $n$, the identity
	\begin{align}
		h_n[\X\Y] = \sum_{\mu\vdash n} F_\mu(\X) G_\mu(\Y), &  & \text{or equivalently} &  & e_n[\X\Y] = \sum_{\mu\vdash n} F_\mu(\X)\, \omega \left( G_\mu(\Y) \right),
	\end{align}
	holds if and only if $\{F_\mu(\X)\}_\mu$ and $\{G_\mu(\X)\}_\mu$ are \emph{dual bases} under the Hall scalar product of symmetric functions.
\end{proposition}

Via the RSK correspondence, it follows that the Schur basis is self dual. Here, $\omega$ denotes the linear and multiplicative involution sending $p_n$ to $(-1)^{n-1} p_n$ (and $s_\mu$ to $s_{\mu'}$).

The modified Macdonald polynomials afford the following $(q,t)$-Cauchy identity
\begin{align}\label{Cauchy_qt}
	e_n^* \left[ {\X\Y} \right]  = \sum_{ \mu \vdash n} \frac{\Ht_\mu(\X;q,t) \Ht_\mu(\Y;q,t)}{w_\mu(q,t)},
\end{align}
where $F^\ast (\X)$ stands for $F \left[ \X/M \right] = F \Big[{\textstyle \frac{\X}{(1-q)(1-t)}}\Big]$, and $w_{\mu} = w_\mu(q,t)$ is the polynomial
\begin{align*}
	w_\mu(q,t) \coloneqq \prod_{c\in\mu} \left( q^{a(c)}-t^{\ell(c)+1} \right) \left( t^{\ell(c)}-q^{a(c)+1} \right).
\end{align*}
Just as the Cauchy identity is related to the Hall scalar product, the $(q,t)$-Cauchy identity relates to the following scalar product.
\begin{definition}
	The $(q,t)$-\emph{scalar product} $\langle F,G \rangle_*$ is defined on the basis $\{p_\mu\}_\mu$ of power sum symmetric functions by setting
	\begin{align}
		\langle p_\mu,p_\lambda\rangle_* \coloneqq \delta_{\lambda \mu} (-1)^{n-\ell(\mu)} z_\mu~ p_\mu[(1-q)(1-t)]
	\end{align}
	where the factor $z_\mu = \langle p_\mu, p_\mu \rangle = 1^{m_1(\mu)} m_1(\mu)! \, 2^{m_2(\mu)} m_2(\mu)! \cdots n^{m_n(\mu)} m_n(\mu)!$ comes from the usual Hall scalar product, with $m_i(\mu)$ being the multiplicity of the part $i$ in $\mu$.
\end{definition}
It follows that, for any two symmetric functions $F$ and $G$,  we have
\begin{align}
	\langle F, G \rangle_\ast = \langle F(X), \omega(G)[M\X] \rangle.
\end{align}

Macdonald polynomials are orthogonal with respect to this inner product. Indeed,
\begin{align}
	\langle  \Ht_\lambda, \Ht_\mu \rangle_\ast = \delta_{\lambda \mu} w_\mu.
\end{align}
Observe that \Cref{Cauchy_qt} implies that we have the identities
\begin{align}
	\label{en_to_H}
	e_n^\ast (\X) = \sum_{ \mu \vdash n} \frac{\Ht_\mu(\X;q,t) }{w_\mu(q,t)} \qquad \text{and} \qquad e_n^\ast \left[ {\X\Y}\right] = \nabla_{\Y} e_n \left[ {\X}/{M} \right].
\end{align}
Furthermore, from the evaluation
\begin{align}
	\Ht_\mu[M] = M \, B_\mu(q,t) \, \Pi_\mu(q,t),
\end{align}
and~\Cref{Cauchy_qt}, we get the well-known expansion
\begin{align}
	e_n(\X)= e_n\left[ M\,\frac{\X}{M} \right]=  \sum_{\mu\vdash n}\frac{M \, B_\mu\,\Pi_\mu}{w_\mu}\, \Ht_\mu(\X;q,t).
\end{align}
Equivalently, in terms of the operators of \Cref{Definition:Mac_Operators}, we have
\begin{align}
	\label{Equation:Xi_en}
	e_n(\X) = M \Delta_{e_1} \Pi\, e_n^\ast(\X).
\end{align}

\subsection{Specializations}
Some of the simplest (plethystic) specializations of the modified Macdonald polynomials are as follows:
\begin{align*}
	\Ht_\mu[1] = 1,                                 &  &
	\Ht_\mu[ -1] = (-1)^{|\mu|} T_\mu, \text{ and } &  &
	\Ht_\mu[ -\epsilon] = T_\mu,
\end{align*}
where $\epsilon$ is formally considered to be such that $p_n(\epsilon)=(-1)^n$.
Hence, on symmetric functions of homogeneous degree $n$, the operator $\nabla_{\Y}$ specializes to
\begin{align*}
	\nabla_{1}= \mathrm{Id},     &  &
	\nabla_{-1} = (-1)^n \nabla, \text{ and }&  &
	\nabla_{-\epsilon} = \nabla.
\end{align*}
A richer, similar specialization, in terms of an extra variable $u$, is given by
\begin{align}
	\Ht_{\mu}[1-u] = \prod_{(i,j)\in \mu} (1- q^i t^j\,u), &  & \text{or equivalently} &  & \Ht_{\mu}[1-\epsilon u] = \prod_{(i,j)\in \mu} (1+ q^i t^j\,u).
\end{align}
At $t=1$, the Macdonald polynomial $\Ht_{\mu}$ specializes to
\begin{align}\label{specialization:t_equal_1}
	 & \Ht_{\mu}(\X;q,1) = \frac{\widehat{h}_\mu(\X)}{\widehat{h}_\mu(1)}, &  & \text{and thus} &  & \Ht_{\mu}(\X;0,1) = h_\mu(\X),
\end{align}
where, again, for any symmetric function $F=F(\X)$, we write $\widehat{F}(\X)$ for $ F \big[ {\textstyle \frac{\X}{1-q}} \big]$. Moreover, at $t=1/q$ we get the similar expression
\begin{align}
	\label{specialization:t_equal_1_over_q}
	\Ht_{\mu}(\X;q,1/q) = \frac{\widehat{s}_\mu(\X)}{\widehat{s}_\mu(1)}.
\end{align}
It may be worth recalling that\footnote{Clearly  $\widehat{h}_n(1)$ is also equal to $1/(q;q)_n$,  where $(q;t)_k = (1-q) (1-qt) \cdots(1-qt^{k-1})$ stands for the \emph{Pochammer symbol}. We  sometimes write $(q;q)_\mu$ for the product$ (q;q)_{\mu_1} (q;q)_{\mu_2} \cdots (q;q)_{\mu_k}$. }
\begin{align*}
	\widehat{h}_n(1) & = h_n\Big[{\textstyle \frac{1}{1-q}}\Big] = \prod_{i=1}^n \frac{1}{1-q^i},   \\
	\widehat{e}_n(1) & = e_n\Big[{\textstyle \frac{1}{1-q}}\Big] = \prod_{i=1}^n \frac{q^{i-1}}{1-q^i},
\end{align*}
and more generally,
\begin{align*}
	\widehat{s}_\mu(1) = s_\mu\Big[{\textstyle \frac{1}{1-q}}\Big] = s_\mu(1,q,q^2,\ldots) = \frac{q^{n(\mu)}}{\prod_{c\in\mu} {1-q^{\text{hook}(c)}}},
\end{align*}
where $\text{hook}(c)$ stands for the hook length of $c$ in $\mu$, and $n(\mu)\coloneqq\sum_{(i,j)\in\mu} (j-1)$.

Well-known operators may be obtained as specializations of $\nabla_{\Y}$. Recall that, for any given symmetric function $F$, one considers the $\Delta$ operators \cite{Bergeron-Garsia-Haiman-Tesler-Positivity-1999} defined by
\begin{equation}
	\Delta_F \Ht_\mu \coloneqq F[B_\mu(q,t)]\, \Ht_\mu.
\end{equation}
We then observe that
\begin{align}
	\nabla_{1-\epsilon u}= \sum_{k\geq 0} u^k \Delta_{e_k}.
\end{align}
Applying (independently) two such specializations to $\nabla_\ast$, we clearly get
\begin{align}
	\nabla_{1-\epsilon u} \nabla_{1-\epsilon v}= \sum_{k,j \geq 0} u^k v^j \Delta_{e_ke_j},
\end{align}
It follows by iteration that we can obtain any of the operators $\Delta_{e_\lambda}$, by taking suitable coefficients in compositions of operators $\nabla_\ast$. Thus, any of the operators $\Delta_F$ may also be obtained by simply taking suitable linear combinations of these compositions. Observe that, on non-constant symmetric functions, we can also obtain the operator $\Pi$ from the specialization $\lim_{u\to 1} (1-u)^{-1} \nabla_{1-u}$. Another interesting observation along these lines concerns the operator $\Xi$ \cite{IraciRomero2022DeltaTheta}.
\begin{definition}
	Define the linear operator $\Xi \colon \Lambda \rightarrow \Lambda$ by setting, for $F \in \Lambda$,
	\begin{align*}
		\Xi F(\X) \coloneqq M \Delta_{e_1} \Pi \, F^\ast(\X).
	\end{align*}
\end{definition}
As we will now show, both of the symmetric functions $\Xi(e_\lambda)$ and $\Xi(s_\lambda)$ can be expressed in terms of $\nabla_\ast e_n$. Indeed, we first observe  from \Cref{Equation:Xi_en} that we have \[ \Xi e_n = \sum_{\mu \vdash n} \frac{M B_\mu \Pi_\mu}{w_\mu} \Ht_\mu = e_n. \]
For any $\lambda$, we can now write
\begin{align*}
	\Xi e_\lambda 	& = \sum_{\mu \vdash n} \frac{M B_\mu \Pi_\mu}{w_\mu} \Ht_\mu \langle e_{\lambda}^\ast , \Ht_\mu \rangle_\ast \\
					& = \sum_{\mu \vdash n} \frac{M B_\mu \Pi_\mu}{w_\mu} \Ht_\mu \langle h_\lambda , \Ht_\mu \rangle             \\
					& = \langle h_\lambda, \Ht_\mu \rangle \sum_{\mu \vdash n} \frac{M B_\mu \Pi_\mu}{w_\mu} \Ht_\mu              \\
					& = \langle h_\lambda, \nabla_\ast \rangle e_n
\end{align*}

It is clear that, for any $F$, the operators $\langle F, \nabla_\ast \rangle$ and $\nabla_\ast$ commute. We have thus shown, by linearity and an application of $\nabla_\ast^k$, that
\begin{proposition}
	For any partition $\lambda$, we have
	\begin{align*}
		\nabla_\ast^k \Xi e_\lambda = \langle h_\lambda, \nabla_\ast \rangle \nabla_\ast^k e_n, &  & \text{and} &  &
		\nabla_\ast^k \Xi s_\lambda = \langle s_\lambda, \nabla_\ast \rangle \nabla_\ast^k e_n.
	\end{align*}
\end{proposition}
A consequence of this result is that, in order to compute $\nabla_\ast^k \Xi e_\lambda$ and $\nabla_\ast^k \Xi s_\lambda$, it is enough to compute $\nabla_\ast^{k+1} e_n$ and take a scalar product.

\subsection{Setting $t=1$}
Clearly, the above calculations may be specialized at $t=1$. Observe that, applying the specialization of \Cref{specialization:t_equal_1}, we directly obtain
\[ \left. \nabla_\ast \Ht_\mu(\X;q,t) \right\rvert_{t=1} = \frac{\widehat{h}_\mu}{\widehat{h}_\mu(1)} \otimes \frac{\widehat{h}_\mu}{\widehat{h}_\mu(1)}. \]
The specialization at $t=1$ of the operator $\nabla_\ast$ may thus be considered as an operator itself.
\begin{definition}
	We define $\widehat{\nabla}_\ast \colon \Lambda_{\mathbb{Q}(q)} \rightarrow \Lambda_{\mathbb{Q}(q)} \otimes \Lambda_{\mathbb{Q}(q)}$, as
	\begin{align}\label{definition:nabla_t_equal_1}
		\widehat{\nabla}_\ast \widehat{h}_\mu = \frac{1}{\widehat{h}_\mu(1)} \widehat{h}_\mu \otimes \widehat{h}_\mu.
	\end{align}
\end{definition}

It is now clear that, for any symmetric function $F$ with no poles at $t=1$, we have
\begin{align}
	\left. \left( \nabla_\ast F \right) \right\rvert_{t=1} = \widehat{\nabla}_\ast \left( F \rvert_{t=1} \right).
\end{align}

Applying Cauchy's formula with the pair of dual basis $(m_\mu,h_\nu)$, we get
\begin{align}
	e_n(\X) & = e_n \left[ (1-q)\,\frac{\X}{1-q} \right] \\
	        & = \sum_{\mu \vdash n} f_\mu[1-q]\, h_\mu \left[ \frac{\X}{1-q} \right] \\
	        & = \sum_{\mu \vdash n} f_\mu[1-q]\, \widehat{h}_\mu(\X),
\end{align}
and therefore
\begin{align}
	\label{eq:firstexpansion}
	\widehat{\nabla}_\ast^k e_n = \sum_{\mu \vdash n} f_\mu[1-q]\ \widehat{h}(1)\ \left( \frac{\widehat{h}_\mu}{\widehat{h}(1)} \right)^{\otimes k+1}
\end{align}
The monomial expansion of $\widehat{h}_\mu(\X)$ may be calculated, for example, by specializing Haglund, Haiman, and Loehr's formula for the modified Macdonald polynomials \cite{Haglund-Haiman-Loehr}.
\[ \left. \langle \Ht_\mu, h_\lambda \rangle  \right\rvert_{t=1} = \frac{1}{\widehat{h}_\mu(1)}  \langle \widehat{h}_\mu, h_\lambda \rangle = \sum_{m(\alpha) =  1^{\lambda_1} 2^{\lambda_2} \cdots} q^{\revmaj_\mu( \alpha)}. \]
For any composition $\beta \in \R(\mu)$, taking the $(k+1)$-tensor of the formula above, we have
\begin{equation*}
	\left( \frac{\widehat{h}_\mu}{\widehat{h}(1)} \right)^{\otimes k+1} = \sum_{w \in (\mathbb{N}_+^{k+1})^n} q^{\revmaj_\beta(w)} \XX^w
\end{equation*}
where if $w = (w_1, \dots, w_n)$ with $w_i=(w_{i0},\dots, w_{ik}) \in \mathbb{N}_+^{k+1}$,
\begin{equation*}
	\revmaj_\beta(w) = \sum_{i=0}^{k} \revmaj_\beta( w_{1i}, \dots, w_{ni}),
\end{equation*}
and \[ \XX^w = \bigotimes_{i=0}^{k} \X^{w_{\ast i}} = \prod_{i=0}^{k} \prod_{j=1}^n x_{i,w_{ji}} \quad \text{ with } \quad w_{\ast i} = (w_{1i}, w_{2i}, \dots w_{ni}). \]
The only term in \Cref{eq:firstexpansion} left to describe is the specialization $f_\mu[1-q]$. To see what this is, we first see that
\[ e_n = \sum_{\mu \vdash n} \left( \frac{MB_\mu \Pi_\mu}{w_\mu} \right) \Big|_{t=1}\  \frac{\widehat{h}_\mu}{\widehat{h}_\mu(1)}  = \sum_{\mu \vdash n} f_\mu[1-q]\, \widehat{h}_\mu. \]
It can be shown, as done in \cite{IraciRomero2022DeltaTheta}, that at $t=1$ one has
\[ \left. \left( \frac{M B_\mu \Pi_\mu}{w_\mu} \right)\right\rvert_{t=1} = (-1)^{n-\ell(\mu)}\ \widehat{h}_\mu(1)\sum_{\alpha \in R(\mu)} (1-q^{\alpha_1}) = \widehat{h}_\mu(1)\, f_\mu[1-q], \]
meaning that
\begin{align}
	f_\mu[1-q] = (-1)^{n-\ell(\mu)} \sum_{\alpha \in R(\mu)} (1-q^{\alpha_1}),
\end{align}
as was also seen in \cite{Hicks-Romero-2018} using summation formulas.

Lastly, from the Cauchy identity, it follows that
\begin{align*}
	\widehat{h}_\mu & =  \sum_{ \eta \vdash n } e_\eta \sum_{ \vec{v} \in \PR(\eta,\mu)} \widehat{f}_{\vec{v}}(1),
\end{align*}
where, for convenience, we set \[ f_{\vec{v}} = \prod_{i=1}^{\ell(\vec{\nu})} f_{\nu^i}. \]

We can now combine these identities to get our preliminary expansion.
\begin{proposition}
	For any $n$,
	\begin{align*} 
		\widehat{\nabla}_\ast^k e_n = \sum_{\eta \vdash n} e_\eta \otimes \left( \sum_{\mu \vdash n} (-1)^{n-\ell(\mu)} \sum_{\beta \in \R(\mu)} (1-q^{\beta_1}) \sum_{w \in (\mathbb{N}^{k})^n} q^{\revmaj_\beta(w)}\, \XX^w \sum_{\vec{\nu} \in \PR(\eta,\mu)} \widehat{f}_{\vec{\nu}} \right)
	\end{align*}
	where
	\[ \XX^w =  \prod_{i=1}^n \prod_{j=1}^k x_{j,w_{ij}}. \]
\end{proposition}

\begin{definition}
	\label{def:DExpression}
	Given a set of words $W$, a function $\rho \colon W \rightarrow A[q]$ (with $A$ some given ring), and a composition $\eta \vDash n$,  we define the rational function
	\begin{align*}
		D_{\eta,W}^\rho(q) \coloneqq  \sum_{\mu \vdash n} (-1)^{n-\ell(\mu)} \sum_{\beta \in \R(\mu)} (1-q^{\beta_1})
		\sum_{w \in W} \rho(w) \sum_{\vec{\nu} \in \PR(\eta,\mu)} f_{\vec{\nu}}[1/(1-q)].
	\end{align*}
\end{definition}

In \Cref{definition:WV} we will consider specific choices for $\rho$ that will be said to \emph{look right}. For these, we will show that the above rational function is not only a polynomial, but it is given by the area enumeration of a class of $\rho$-compatible labeled parallelogram polyominoes.



	

\section{Different weights}
\label{section:weights}

In this section we extend the results of \cite{IraciRomero2022DeltaTheta}*{Sections~7,10} to more general weights.
Through this section, we will fix a set of words $W \subseteq \bigcup_{n \in \mathbb{N}} X^n$ in the alphabet $X$, and we let $A$ be some fixed ring.
\begin{definition}
	\label{looksright}
	Let $X$ be any alphabet of variables, and let $W \subseteq \bigcup_{n \in \mathbb{N}} X^n$ be a set of words in the alphabet $X$. Let $A$ be any ring. We say that a function $\rho \colon W \rightarrow A[q]$ \emph{looks right} if there exist a weight function $\wt \colon X \rightarrow A$ and a local weight function $\lwt \colon W \times \mathbb{N} \rightarrow \mathbb{N}$ giving
	\[ \rho(w) = \prod_{i=1}^{\ell(w)} q^{\lwt(w,i) (\ell(w)-i)} \wt(w_i) \in A[q], \]
	where $\lwt(w,i)$ depends only on $w_i$ and $w_{i+1}$.
\end{definition}
For example, if $X = \mathbb{N}^k$ and $W = \bigcup_{n \in \mathbb{N}} X^n$ is the set of all words with letters in $X$, then $\revmaj$ is a statistic which looks right with local weight function $\lwt(w,i) = r$ if $w$ has an $r$-descent at $i$.

\begin{remark}
	Let $S \colon W \rightarrow 2^\mathbb{N}$ such that $S(w) \subseteq \{1, 2, \dots, \ell(w)\}$ be any function that picks a susbet of the indices of $w$. For $X = \mathbb{N}$, $A = \mathbb{Z}$, $\wt = 1$, and $\lwt(w,i) = 1$ if $i \in S(w)$ and $0$ otherwise, up to taking the $q$-logarithm, \Cref{looksright} reduces to \cite{IraciRomero2022DeltaTheta}*{Definition~10.1}.
\end{remark}

From now on, we assume that $W$ is equipped with a function $\rho$ that looks right.

\begin{definition}
	\label{definition:WV}
	A \emph{word vector} $\vec{w} \in WV(W)$ is a sequence of words $(w^1,\dots, w^r)$ whose concatenation $w^1\cdots w^r $ is in $W$. If $\ell(w^i) = \beta_i$, then we write $\vec{w} \in \WV(W,\beta)$. If $\rho$ looks right on subwords of elements in $W$, then we may set
	\[ \rho(\vec{w}) = \prod_{i=1}^r \rho(w^i), \] and say that $\rho$ looks right on $\WV(W)$.
\end{definition}
For instance, if $X= \mathbb{N}^3$, then
$( (142,221,421),(432),(212,670)) \in \WV(W, (3,1,2)) $ is a word vector whose entries are words of lengths $3,1$ and $2$ respectively; and the letters in each word are tuples of length $3$. Here, we can choose $W$ to be $(\mathbb{N}^3)^6$.

\begin{definition}\label{LCCTDefinition}
	Let $\eta$ be a partition of  $n$, and $W \subseteq \bigcup_{n \in \mathbb{N}} X^n$ be a set of words in the alphabet $X$. A \emph{sequence of $W$-labeled column-composition tableaux} of type $\eta$ is a tuple of labeled column-composition tableaux $(C^i, w^i)_{1 \leq i \leq r}$ such that, with $\beta = (\beta_1, \dots, \beta_r)$ and $\beta_i = \ell(C^i)$, we have:
	\begin{enumerate}
		\item $C^1 \in \overline{\CC}_{\nu^1}$ and $C^i \in \CC_{\nu^i}$ for $i>1$, for some $\vec{\nu} \in \PR(\eta, \beta)$;
		\item $\vec{w} = (w^1, \dots, w^r) \in \WV(W, \beta)$.
	\end{enumerate}
	We denote by $\LC_{W,\eta}$ the set of sequences of $W$-labeled column-composition tableaux of type $\eta$.
\end{definition}

In other words, a sequence of $W$-labeled column-composition tableaux is a tuple of column-composition tableaux of sizes $\beta_1, \dots, \beta_r$ with $c_1(C^1) = 0$, where under column $j$ of $C^i$ we place an element of $X$, $w^i_j$. When read from left to right, the labels under the column give a word $w^1\cdots w^r \in W$. If the vertical bars in each $C^i$ gives a composition rearranging to $\nu^i$, then $\nu^1,\dots,\nu^{r}$ has parts rearranging to $\eta$. See \Cref{figure:SExample} for an example of such a sequence with $\beta = (2,5,2)$. 

For $T = (T_i)_{1 \leq i \leq r} \in \LC_{W,\eta}$, we set $w(T_i) \coloneqq w^i$. Given a function  $\rho$ on word vectors from $W$ that looks right, for $T \in \LC_{W, \eta}$ we set
\begin{align*}
	\weight(T_i) \coloneqq q^{\lvert C^i \rvert} \rho(w^i), &  & \weight(T) \coloneqq \prod_{i=1}^{\ell(\eta)} \weight(T_i), &  & \text{ and } &  & \sign(T) \coloneqq (-1)^b,
\end{align*}
where $b$ stands for the number of vertical bars in $C^1,\dots, C^{\ell(\eta)}$. An example of this weight and sign computation is given in \Cref{figure:SExample}.



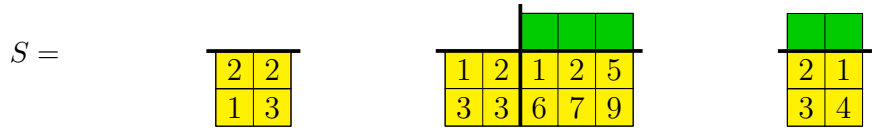
\begin{figure}[ht]
	\begin{align*}
		\begin{tikzpicture}[ xscale=-1]
			\draw (.5,1) node {$S =$};
			\draw (0,0) circle (0cm);
		\end{tikzpicture}                    &  &
		\begin{tikzpicture}[xscale=-1,scale=.5]
			\Yboxdim{1cm}
			\ylw
			\tyng(0cm,-2cm,2,2)
			\draw[line width = .5mm] (-.25,1-1) -- (2.25,1-1);
			\begin{large}
				\draw (.5,.5-1) node {$2$};
				\draw (1.5,.5-1) node {$2$};
				\draw (.5,.5-2) node {$3$};
				\draw (1.5,.5-2) node {$1$};
			\end{large}
		\end{tikzpicture} &  &
		\begin{tikzpicture}[xscale=-1,scale=.5]
			\Yboxdim{1cm}
			\tyng(0cm,0cm,3)
			\ylw
			\tyng(0cm,-2cm,5,5)
			\draw[line width = .5mm] (-.25,1-1) -- (5.25,1-1);
			\draw[line width = .5mm] (3, -2) -- (3,1.25);
			\begin{large}
				\draw (.5,.5-1) node {$5$};
				\draw (1.5,.5-1) node {$2$};
				\draw (2.5,.5-1) node {$1$};
				\draw (3.5,.5-1) node {$2 $};
				\draw (4.5,.5-1) node {$1$};
				\draw (.5,.5-2) node {$9$};
				\draw (1.5,.5-2) node {$7$};
				\draw (2.5,.5-2) node {$6 $};
				\draw (3.5,.5-2) node {$3$};
				\draw (4.5,.5-2) node {$3$};
			\end{large}
		\end{tikzpicture}
		                                                   &  &
		\begin{tikzpicture}[xscale=-1,scale=.5]
			\Yboxdim{1cm}
			\tyng(0cm,0cm,2)
			\ylw
			\tyng(0cm,-2cm,2,2)
			\begin{large}
				\draw[line width = .5mm] (-.25,1-1) -- (2.25,1-1);
				\draw (.5,.5-1) node {$1$};
				\draw (1.5,.5-1) node {$2$};
				\draw (.5,.5-2) node {$4$};
				\draw (1.5,.5-2) node {$3 $};
			\end{large}
		\end{tikzpicture}          	\end{align*}
	\caption{An example of $S = (S_1,S_2,S_3) \in \LC_{W,\eta}$ where $W$ is the family of words with letters in $\mathbb{N}^2$ and $\eta = (3,2,2)$. We have $\sign(S) = -1$.
With $\lwt(w,i)$ being the number of ascents at column $i$ and $wt( (w_{i1},w_{i2}))=y_{w_{i1}} z_{w_{i2}}$, we have $\weight(S_1) = q^0 \cdot q^1 y_2^2 z_1 z_3$, $\weight(S_2) = q^3 \cdot q^{7+6} y_1^2 y_2^2 y_5 z_3^2 z_6 z_7 z_9$, and $\weight(S_3) = q^2 \cdot q^{1} y_1 y_2 z_3 z_4$.
	}
	\label{figure:SExample}
\end{figure}

We now consider, for a fixed partition $\eta$, the rational function
\begin{align*}
	D_{W, \eta}(q) & \coloneqq \sum_{\beta \vDash n} \sum_{\vec{w} \in \WV(W,\beta)} \sum_{\vec{\nu} \in \PR(\eta,\beta)} \rho(\vec{w}) (-1)^{n-\ell(\beta)} (1-q^{\beta_1}) f_{\vec{\nu}}[1/(1-q)] \\
	               & = \sum_{T \in \LC_{W,\eta}} \sign(T) \weight(T).
\end{align*}
We have the following analogue of \cite{IraciRomero2022DeltaTheta}*{Definition~7.1}, which generates a sign-reversing involution on the set $\LC_{W,\eta}.$

\begin{definition}
	\label{Definition:split}
	Let $S = (C, w)$ be one of the possible labeled column-composition tableaux appearing in a sequence $T \in \LC_{W, \eta}$, and suppose that $S$ has at least one bar. Then we say that $S$ can \emph{split} and define $\spl(S) = (S^1, S^2)$, where $S^1 = (C^1, w^1)$ is the portion of $S$ occurring before the first vertical bar, and $S^2 = (C^2, w^2)$ is obtained from the portion of $S$ after the first vertical bar by adding $\sum_{i=1}^{\ell(C^1)} \lwt(w, i)$
	cells to each column. The split map is weight-preserving and sign-reversing.
\end{definition}

\begin{proposition}
	\label{weightpreserving}
	If $\spl(S) = (S^1, S^2)$, then $\weight(S) = \weight(S^1) \cdot \weight(S^2)$.
\end{proposition}

\begin{proof}[\bf Proof] Let $S = (C, w)$, with $\ell(C) = n$.
	Suppose $\spl(S) = (S^1, S^2)$ with $S^1 = (C^1, w^1)$ and $S^2 = (C^2, w^2)$, and let $v = \ell(C^1)$.

	By definition, the weight has two components, one coming from the total size, and one coming from $\rho$.
	Let $d = \sum_{i=1}^{v} \lwt(w, i)$. By definition of $\spl$, the number of cells above $S^1$ stays the same, while the number of cells above $S^2$ increases by $d \cdot \ell(C_2) = d(n-v)$, so the first component of the total weight increases by $q^{d(n-v)}$.
	By definition of $\rho$, we have
	\begin{align*}
		\rho(w) & = \prod_{i=1}^{n} q^{\lwt(w,i) (n-i)} \wt(w_i)                                                                                                 \\
		        & = \prod_{i=1}^{v} q^{\lwt(w,i) (n-i)} \prod_{i=v+1}^{n} q^{\lwt(w,i) (n-i)} \prod_{i=1}^{n} \wt(w_i)                                           \\
		        & = \prod_{i=1}^{v} q^{\lwt(w,i) (n-v+v-i)} \prod_{i=1}^{n-v} q^{\lwt(w^2,i) (n-v-i)} \prod_{i=1}^{n} \wt(w_i)                                   \\
		        & = \prod_{i=1}^{v} q^{\lwt(w,i) (n-v)} \prod_{i=1}^{v} q^{\lwt(w^1,i) (v-i)} \prod_{i=1}^{n-v} q^{\lwt(w^2,i) (n-v-i)} \prod_{i=1}^{n} \wt(w_i) \\
		        & = \prod_{i=1}^{v} q^{\lwt(w,i) (n-v)} \rho(w^1) \rho(w^2)                                                                                      \\
		        & = q^{d(n-v)} \rho(w^1) \rho(w^2)
	\end{align*}
	so the second component of the total weight decreases by $q^{d(n-v)}$. The two changes cancel out and so the weight is preserved, as desired.
\end{proof}

\begin{lemma}
	Let $S^1 = (C^1, w^1)$, $S^2 = (C^2, w^2)$ be two $W$-labeled column-composition tableaux. There exists an $S$ such that $\spl(S) = (S^1, S^2)$ if and only if
	\begin{equation}
		\label{eq:join}
		c_1(S^2) \geq c_\ell(S^1) + \sum_{j = 1}^{\ell(C^1)} \lwt(w^1w^2, j)
	\end{equation}
	If such an $S$ exists, then it is unique; we say that $S^1$ can join $S^2$ and set $\join(S^1, S^2) = S$.
\end{lemma}

\begin{proof}[\bf Proof]
	If such an $S$ exists, then \Cref{eq:join} holds by construction. Suppose that \Cref{eq:join} holds. Then we can define $S$ as the labeled column composition tableau obtained by decreasing the size of each column of $S^2$ by $\sum_{j = 1}^{\ell(C^1)} \lwt(w^1w^2, j)$ and then concatenating it to $S^1$, also concatenating their words. \Cref{eq:join} ensures that the result is still a column-composition tableau. It is now immediate that $\spl(S) = (S^1, S^2)$ and that such an $S$ is unique.
\end{proof}

The following lemma is crucial to ensure that our sign-reversing, weight-preserving involution is well-defined.

\begin{lemma}
	\label{splitlemma}
	Let $S^1$, $S$ be two $W$-labeled column-composition tableaux, and let $\spl(S) = (S^2, S^3)$. Then $S^1$ can join $S^2$ if and only if it can join $S$.
\end{lemma}

\begin{proof}[\bf Proof]
	By construction, $c_1(S^2) = c_1(S)$, so \Cref{eq:join} holds for $S^1$ and $S^2$ if and only if it holds for $S^1$ and $S$.
\end{proof}

Again, for a fixed set of words $W$ and a partition $\eta$, we can now define our weight-preserving, sign-reversing involution as follows.
\begin{definition}
	Given $T = (T_1,\dots, T_r) \in \LC_{W,\eta}$, define $\psi(T)$ by the following process:
	\begin{enumerate}
		\item if $r=0$, then $\psi(T) = T$;
		\item if $T_1$ can split, then $\psi(T) = (\spl(T_1), T_2,\dots, T_r)$;
		\item if $T_1$ cannot split and $T_1$ can join $T_2$, then $\psi(T) = (\join(T_1, T_2), T_3,\dots, T_r)$;
		\item otherwise we inductively define $\psi(T) = (T_1, \psi(T_2, \dots, T_r))$.
	\end{enumerate}
See \Cref{figure:splitSExample} for an example of the image of $S$ in \Cref{figure:SExample}.
\end{definition}

\begin{figure}[ht]
	\begin{align*}
		\begin{tikzpicture}[ xscale=-1]
			\draw (.5,1) node {$\psi(S) =$};
			\draw (0,0) circle (0cm);
		\end{tikzpicture}                    &  &
		\begin{tikzpicture}[xscale=-1,scale=.5]
			\Yboxdim{1cm}
			\ylw
			\tyng(0cm,-2cm,2,2)
			\draw[line width = .5mm] (-.25,1-1) -- (2.25,1-1);
			\begin{large}
				\draw (.5,.5-1) node {$2$};
				\draw (1.5,.5-1) node {$2$};
				\draw (.5,.5-2) node {$3$};
				\draw (1.5,.5-2) node {$1$};
			\end{large}
		\end{tikzpicture} &  &
		\begin{tikzpicture}[xscale=-1,scale=.5]
			\Yboxdim{1cm}
			\ylw
			\tyng(0cm,-2cm,2,2)
			\draw[line width = .5mm] (-.25,1-1) -- (2.25,1-1);
			\begin{large}
				\draw (.5,.5-1) node {$2$};
				\draw (1.5,.5-1) node {$1$};
				\draw (.5,.5-2) node {$3$};
				\draw (1.5,.5-2) node {$3$};
			\end{large}
		\end{tikzpicture}
		                                                   & &
		\begin{tikzpicture}[xscale=-1,scale=.5]
			\Yboxdim{1cm}
			\tyng(0cm,0cm,3,3,3)
			\ylw
			\tyng(0cm,-2cm,3,3)
			\draw[line width = .5mm] (-.25,1-1) -- (3.25,1-1);
			\begin{large}
				\draw (.5,.5-1) node {$5$};
				\draw (1.5,.5-1) node {$2$};
				\draw (2.5,.5-1) node {$1$};
				\draw (.5,.5-2) node {$9$};
				\draw (1.5,.5-2) node {$7$};
				\draw (2.5,.5-2) node {$6 $};
			\end{large}
		\end{tikzpicture}
		                                                   &  &
		\begin{tikzpicture}[xscale=-1,scale=.5]
			\Yboxdim{1cm}
			\tyng(0cm,0cm,2)
			\ylw
			\tyng(0cm,-2cm,2,2)
			\begin{large}
				\draw[line width = .5mm] (-.25,1-1) -- (2.25,1-1);
				\draw (.5,.5-1) node {$1$};
				\draw (1.5,.5-1) node {$2$};
				\draw (.5,.5-2) node {$4$};
				\draw (1.5,.5-2) node {$3 $};
			\end{large}
		\end{tikzpicture}         
		 	\end{align*}
	\caption{An example of $\psi(S)$ for $S$ in \Cref{figure:SExample}.
		}
	\label{figure:splitSExample}
\end{figure}
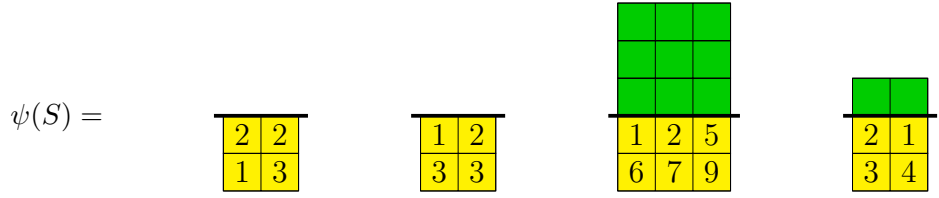

For $T \in \LC_{W,\eta}$ with no vertical bars, we have that for any $i$, $c_j(T_i)$ does not depend on $j$, as having no vertical bars implies that the number of cells above the base is constant. Let us write $c(T_i) \coloneqq c_j(T_i)$ in this case. Note that $c(T_1) = 0$ by definition. We have the following analogue of \cite{IraciRomero2022DeltaTheta}*{Theorem~7.8}.

\begin{proposition}
	Let $U_{W,\eta} $ be the set of $T \in \LC_{W,\eta}$ such that $T$ has no vertical bars, and for all $i$,
	\begin{align*}
		c(T_{i+1}) < c(T_i) + \sum_{j = 1}^{\ell(C^i)} \lwt \Big(w, \sum_{k=1}^{i-1} \ell(C^k) + j \Big).
	\end{align*}
	Then $D_{W,\eta}(q)$ is given by the weight-sum over $U_{W,\eta}$:
	\[
		D_{W,\eta}(q) = \sum_{T \in U_{W,\eta}} {\weight(T)}.
	\]
\end{proposition}

\begin{proof}[\bf Proof]
	Same as \cite{IraciRomero2022DeltaTheta}*{Theorem~7.8}.
\end{proof}

We can now describe the fixed points in terms of parallelogram polyominoes.

\begin{definition}
	\label{def:polyomino}
	For integers $m$ and $n$, a \emph{parallelogram polyomino} of size $m \times n$ is a pair of lattice paths $(P,Q)$ from $(0,0)$ to $(m,n)$, consisting of unit north steps and east steps, such that $P$ (the \emph{top path}) lies always strictly above $Q$ (the \emph{bottom path}), except on the endpoints.
\end{definition}

The \emph{area} of a parallelogram polyomino of size $m \times n$ is defined as \[ \area(P,Q) \coloneqq (\#\text{ of lattice cells between $P$ and $Q$}) - (m+n-1). \]
Since the two paths $P$ and $Q$ do not touch between the endpoints, $m+n-1$ is the minimal number of unit cells between them. One also sees that $\area(P,Q)$ counts the number of lattice cells between $P$ and $Q$ which do not touch $Q$.

\begin{figure}[!ht]
	\begin{center}
		\includegraphics{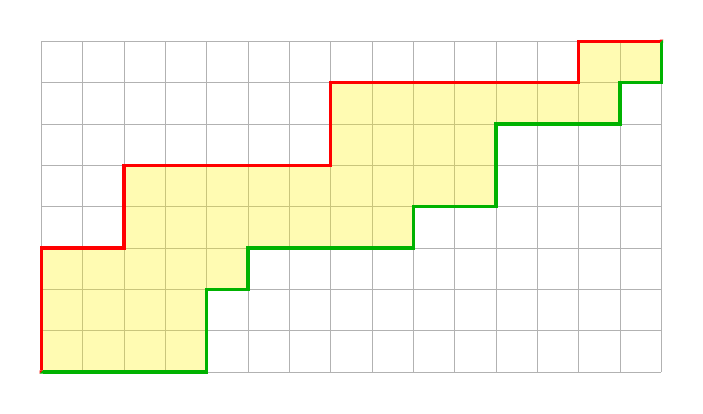}
	\end{center}
	\caption{A parallelogram polyomino with area $20$.}
\end{figure}

\begin{definition}
	\label{WLPP-definition}
	An $m \times n$, \emph{$W$-labeled parallelogram polyomino} is a triple $(P,Q,w)$ such that $(P,Q)$ is an $m \times n$ parallelogram polyomino, and $w \in W \cap X^n$ is a word in $W$ of lenght $n$.
\end{definition}

\begin{definition}
	Given $\rho \colon W \rightarrow A[q]$ that looks right, so that
	\[
		\rho(w) = \prod_{i=1}^{\ell(w)} q^{\lwt(w,i) (\ell(w)-i)} \wt(w^i) \in A[q]
	\]
	for some appropriate weight functions $\wt$ and $\lwt$, we say that a parallelogram polyomino $(P,Q,w)$ is \emph{$\rho$-compatible} if $Q$ contains exactly $\lwt(w,i) + \delta_{i,1}$ east steps on the line $y=i-1$.
\end{definition}

\begin{definition}
	Given $\eta \vdash n$, we say that a $W$-labeled parallelogram polyomino $(P,Q,w)$ has \emph{type $\eta$} if the lengths of the maximal streaks of  north steps of $P$ form a composition $\alpha \in \R(\eta)$.
\end{definition}

Let $\PP_{W,\eta}$ be the set of all $W$-labeled parallelogram polyominoes of type $\eta$ that are $\rho$-compatible. We have the following.

\begin{proposition}
	\label{proposition:phi}
	There is a bijection $\varphi \colon U_{W, \eta} \leftrightarrow \PP_{W,\eta}$ such that \[ \weight(T) = q^{\area(\varphi(T))} \prod_{i=1}^{\ell(w)} \wt(w_i). \]
\end{proposition}

\begin{proof}[\bf Proof]
	For $T = (C^i, w^i)_{1 \leq i \leq r} \in U_{W,\eta}$, we need to define a triple \[ \varphi(T) = (P(T),Q(T),w(T)) \] corresponding to the polyomino and its labels. An example of $\varphi$ in the context of \Cref{figure:SExample} is given in \Cref{figure:PolyominoBijectionExample}.
	
	For $1 \leq i \leq r$, let \[ \beta_i = \ell(C^i), \qquad \Sigma_i = \sum_{j=1}^{i} \beta_i, \qquad s_i = c(T_i) + \sum_{j = 1}^{\beta_i} \lwt(w, \Sigma_{i-1} + j) - c(T_{i+1}); \] the latter is guaranteed to be strictly positive because $T \in U_{W, \eta}$. We define
	\begin{itemize}
		\item $P = P(T) \coloneqq \left( \prod_{i=1}^r N^{\beta_i} E^{s_i} \right) E$
		\item $Q = Q(T) \coloneqq E \left( \prod_{i=1}^{n} E^{\lwt(w,i)} N \right)$
		\item $w = w(T) \coloneqq w^1 w^2 \cdots w^r$
	\end{itemize}
	where the product denotes the ordered concatenation of strings. 
	We claim that $\varphi$ is the desired bijection.

	First we need to show that it is well-defined, i.e.\ $(P,Q,w)$ is a $W$-labeled parallelogram polyomino of type $\eta$ that is $\rho$-compatible. Indeed, the lengths of the maximal streaks of  north steps of $P$ form a composition $\alpha \in \R(\eta)$, $Q$ satisfies the $\rho$-compatibility condition by construction, and $w$ is a $W$-labeling. We need to show that $P$ and $Q$ have the same size and that $P$ always lies strictly above $Q$.

	The total height of $P$ is $\sum_{i=1}^r \beta_i = n$, and the total height of $Q$ is $n$, so they agree. The total width of $P$ is
	\begin{align*}
		1 + \sum_{i=1}^r s_i & = 1 + \sum_{i=1}^r \left( c(T_i) + \sum_{j = 1}^{\beta_i} \lwt(w, \Sigma_{i-1} + j) - c(T_{i+1}) \right) \\
		                     & = 1 + \sum_{i=1}^r (c(T_i) - c(T_{i+1})) + \sum_{i=1}^r \sum_{j=1}^{\beta_i} \lwt(w, \Sigma_{i-1} + j)   \\
		                     & = 1 + c(T_1) + \sum_{i=1}^n \lwt(w, i)                                                                   \\
		                     & = 1 + \sum_{i=1}^n \lwt(w, i),
	\end{align*}
	which is exactly the total width of $Q$, as expected. To show that $P$ lies always strictly above $Q$, it is enough to check that, for $1 \leq k \leq r$, we have \[ \sum_{i=1}^k s_i \leq \sum_{i=1}^{\Sigma_k} \lwt(w, i). \]
	But this is immediate because the difference of the two quantities is $c(T_{i+1}) \geq 0$.

	Finally, we have to show that \[ \area(P,Q) = \sum_{i=1}^r \left( \lvert C^i \rvert + \sum_{j=1}^{\beta_i} \lwt(w^i, j)(\beta_i - j) \right). \]
	Let $\area(P)$ (and $\area(Q)$) be the number of whole cells between $P$ (respectively $Q$) and the bottom-right semi-perimeter of the rectangle containing $(P,Q)$. Then $\area(P,Q) = \area(P) - \area(Q) - (m+n-1)$, where $m$ is the width of the path.

	We can compute the first terms factors as follows: for each  north step of the relevant path, we count the number of east steps that follow it and take the sum of these quantities over all  north steps. In this way, for the first summand, we get
	\begin{align*}
		\area(P) & = \sum_{k=1}^r \beta_k \left( 1 + \sum_{i=k}^r s_i \right)                                                                    \\
		         & = \sum_{k=1}^r \beta_k \sum_{i=k}^r \left( 1 + c(T_i) + \sum_{j = 1}^{\beta_i} \lwt(w, \Sigma_{i-1} + j) - c(T_{i+1}) \right) \\
		         & = \sum_{k=1}^r \beta_k \sum_{i=k}^r \left( 1 + c(T_i) - c(T_{i+1}) \right)
		+ \sum_{k=1}^r \beta_k \sum_{i=k}^r \sum_{j = 1}^{\beta_i} \lwt(w, \Sigma_{i-1} + j)                                                     \\
		         & = \sum_{k=1}^r \beta_k \left( c(T_k) + 1 \right) + \sum_{k=1}^r \beta_k
		\sum_{i=k}^r \sum_{j = 1}^{\beta_i} \lwt(w, \Sigma_{i-1} + j)                                                                            \\
		         & = n + \sum_{k=1}^r \lvert C^k \rvert + \sum_{i=1}^r \sum_{j = 1}^{\beta_i} \sum_{k=1}^i \beta_k \lwt(w, \Sigma_{i-1} + j).
	\end{align*}
	For the second summand, we can write
	\begin{align*}
		\area(Q) & = \sum_{i=1}^n \left( \sum_{j=i+1}^n \lwt(w,j) \right)                                                        \\
		         & = \sum_{j=1}^n (j-1) \lwt(w,j)                                                                                \\
		         & = \sum_{i=1}^r \sum_{j=1}^{\beta_i} \left( \sum_{k=1}^{i-1} \beta_k + j - 1 \right) \lwt(w, \Sigma_{i-1} + j).
	\end{align*}
	Finally, for the last summand, we have
	\begin{align*}
		m+n-1 & = \left( 1 + \sum_{i=1}^n \lwt(w, i) \right) + n - 1               \\
		      & = n + \sum_{i=1}^n \lwt(w, i)                                      \\
		      & = n + \sum_{i=1}^r \sum_{j=1}^{\beta_i} \lwt(w, \Sigma_{i-1} + j).
	\end{align*}
	Since $\area(P,Q) = \area(P) - \area(Q) - (m+n-1)$, we have
	\begin{align*}
		\area(P,Q) & = \sum_{k=1}^r \lvert C^k \rvert + \sum_{i=1}^r \sum_{j = 1}^{\beta_i} \lwt(w, \Sigma_{i-1} + j)
		\left( \sum_{k=1}^i \beta_k - \left( \sum_{k=1}^{i-1} \beta_k + j - 1 \right) - 1 \right)                                    \\
		           & = \sum_{k=1}^r \lvert C^k \rvert + \sum_{i=1}^r \sum_{j = 1}^{\beta_i} \lwt(w, \Sigma_{i-1} + j)(\beta_i - j)   \\
		           & = \sum_{i=1}^r \left( \lvert C^i \rvert + \sum_{j = 1}^{\beta_i} \lwt(w, \Sigma_{i-1} + j)(\beta_i - j) \right) \\
		           & =  \sum_{i=1}^r \left( \lvert C^i \rvert + \sum_{j=1}^{\beta_i} \lwt(w^i, j)(\beta_i - j) \right)
	\end{align*}
	because $\lwt(w^i, j) = \lwt(w, \Sigma_{i-1} + j)$ whenever $j \neq \beta_i$; and when $j = \beta_i$ the corresponding summand above is $0$. Therefore, the equality holds and the thesis follows immediately.
\end{proof}

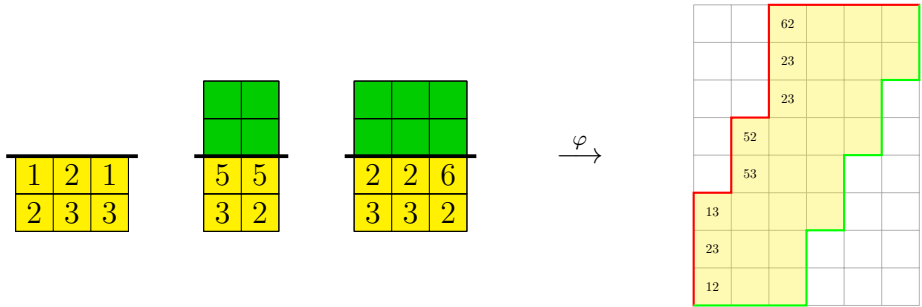
\begin{figure}[ht]
	\begin{align*}
		\begin{tikzpicture}[xscale=-1,scale=.5]
			\draw[draw=none, use as bounding box] (-1,-4) rectangle (3,4);
			\Yboxdim{1cm}
			\ylw
			\tyng(0cm,-2cm,3,3)
			\draw[line width = .5mm] (-.25,1-1) -- (3.25,1-1);
			\begin{large}
				\draw (.5,.5-1) node {$1$};
				\draw (1.5,.5-1) node {$2$};
				\draw (2.5,.5-1) node {$1$};
				\draw (.5,.5-2) node {$3$};
				\draw (1.5,.5-2) node {$3$};
				\draw (2.5,.5-2) node {$2$};
			\end{large}
		\end{tikzpicture}
		\begin{tikzpicture}[xscale=-1,scale=.5]
			\draw[draw=none, use as bounding box] (-1,-4) rectangle (3,4);
			\Yboxdim{1cm}
			\tyng(0cm,0cm,2,2)
			\ylw
			\tyng(0cm,-2cm,2,2)
			\draw[line width = .5mm] (-.25,1-1) -- (2.25,1-1);
			\begin{large}
				\draw (.5,.5-1) node {$5$};
				\draw (1.5,.5-1) node {$5$};
				\draw (.5,.5-2) node {$2$};
				\draw (1.5,.5-2) node {$3$};
			\end{large}
		\end{tikzpicture}
		\begin{tikzpicture}[xscale=-1,scale=.5]
			\draw[draw=none, use as bounding box] (-1,-4) rectangle (4,4);
			\Yboxdim{1cm}
			\tyng(0cm,0cm,3,3)
			\ylw
			\tyng(0cm,-2cm,3,3)
			\draw[line width = .5mm] (-.25,1-1) -- (3.25,1-1);
			\begin{large}
				\draw (.5,.5-1) node {$6$};
				\draw (1.5,.5-1) node {$2$};
				\draw (2.5,.5-1) node {$2$};
				\draw (.5,.5-2) node {$2$};
				\draw (1.5,.5-2) node {$3$};
				\draw (2.5,.5-2) node {$3$};
			\end{large}
		\end{tikzpicture}
		\begin{tikzpicture}[ xscale=1]
			\draw[draw=none, use as bounding box] (-1,-2) rectangle (1.5,2);
			\draw (0,0) node {\raisebox{6pt}{$\overset{\varphi}{{{\longrightarrow}}}$}};
		\end{tikzpicture}
		\includegraphics[scale=0.5]{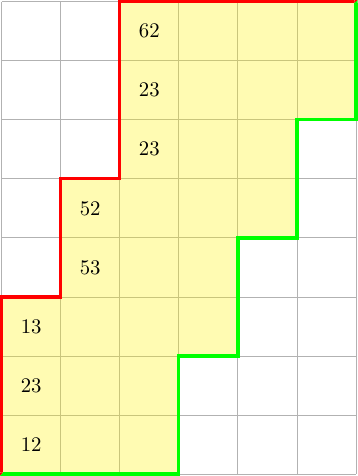}
	\end{align*}
	\caption{An example of the bijection $\varphi$ for a fixed point of $\psi$ in the context of \Cref{figure:SExample}.}
	\label{figure:PolyominoBijectionExample}
\end{figure}

\begin{corollary}
	\label{cor:1}
	For any $\eta \vdash n$, \[ D_{W,\eta} = \sum_{P \in \PP_{W,\eta}} q^{\area(P)} \wt(P), \]
	where if $(w^1,\dots, w^r) = w(\varphi^{-1}(P)) $, we set $\wt(P) = \prod_i \wt(w^i)$.
	In particular, if
	\[ F = \sum_{ \eta \vdash n }  e_\eta \sum_{\beta \vDash n} \sum_{\vec{w} \in \WV(W,\beta)} 
		\sum_{\vec{\nu} \in \PR(\eta,\beta)} \rho(\vec{w}) (-1)^{n-\ell(\beta)} (1-q^{\beta_1}) f_{\vec{\nu}}[1/(1-q)], \]
	then \[ F = \sum_{\eta \vdash n} e_\eta \sum_{ P \in \PP_{W,\eta}} q^{\area(P)} \wt(P). \]
\end{corollary}

\section{Multi-labeled Dyck paths}
\label{section:multi-labeledDyckpaths}

In order to give a combinatorial interpretation of $\nabla_\ast^k e_n$, we must first identify the proper sets and statistics to which we can apply the results from the previous section.

Let $X = \mathbb{N}_+^{k+1}$, $W = \bigcup_{n \in \mathbb{N}} X^n$, and for convenience, we index elements in $X$ from $0$ so that for $w = (w_1,\dots, w_n) \in W$, we have $w_i = (w_{i0}, \dots, w_{ik}) \in \mathbb{N}^{k+1}$.
Let $\lwt(w,i) = \# \{ 0 \leq j \leq k \mid  w_{ij} < w_{i+1,j}\}$, that is, the number of ascents occurring at the same position between the $i^{\ith}$ and the $(i+1)^{\st}$ letter of $w$, and let \[ \wt(w) = \XX^w =  \prod_{i=1}^n \prod_{j=0}^k x_{j,w_{ij}}. \]
Now, recall that
\begin{align*}
	\widehat{\nabla}_\ast^k e_n(\X)
		& = \sum_{\mu \vdash n} \sum_{\beta \in R(\mu)} (-1)^{n-\ell(\mu)} \frac{1-q^{\beta_1}}{ \widehat{h}_\beta(1)}
		\sum_{w \in (\mathbb{N}_+^{k+1})^n}  q^{\revmaj_\beta(w)} \XX^w  \\
		& = D_{W,1^n}(q) = \sum_{P \in \PP_{W,1^n}} q^{\area(P)} \wt(P).
\end{align*}
This gives one way of getting a combinatorial interpretation, but we can also translate this result in terms of what we will call multi-labeled $k^n$-Dyck paths. The explicit bijection to multi-labeled Dyck paths will be given in \Cref{def:iota}, below. But first, we will state the main results.

\begin{definition}
	For positive integers $r$, $k$ and $n$, a  \emph{multilabeled Dyck path} is a pair $(\pi, w) \in \LD_{k^n}^r$, where $\pi \in D_{n}^{k}$ is a $( k n \times n)$ Dyck path (i.e.\ a lattice path that always lie weakly above the diagonal $ky = x$, called the main diagonal), and $w = (w_1, \dots, w_n)$ is a word of $r$-tuples $w_i = (w_{i0}, \dots, w_{i,r-1}) \in \mathbb{N}_+^{r}$ such that, if $w$ has $s$ weak descents at position $i$, then the $i^{\ith}$  north step is followed by at least $s$ east steps. In other words, if the $i^{\ith}$  north step of $\pi$ has $x$-coordinate $\col_\pi(i)$ then \[ \# \{ j \mid w_{i,j} \geq w_{i+1,j} \} \leq \col_\pi(i+1) - \col_\pi(i) \] for all $1 \leq i < n$. The \emph{area} of a multi-labeled Dyck path is defined as is done for rectangular Dyck paths, so that $\area(\pi)$ is the number of whole lattice cells between $\pi$ and the main diagonal (note that the area does not depend on $w$).
	
	When $r = k+1$, we denote the set $\LD_{k^n}^{k+1}$ by $\LD_{k^n}$.
\end{definition}

The notion of $\LD_{k^n}$ is illustrated in \Cref{figure:LD33_paths}. The whole $m$-expansion of $\widehat{\nabla}_{\vertfonce{\Z}} \widehat{\nabla}_{\rouge{\Y}}(e_3(\bleu{\X}))$ may be found in \Cref{Table:nabla_Y_Z}.
\begin{figure}
	\begin{center}
		\includegraphics[scale=0.7]{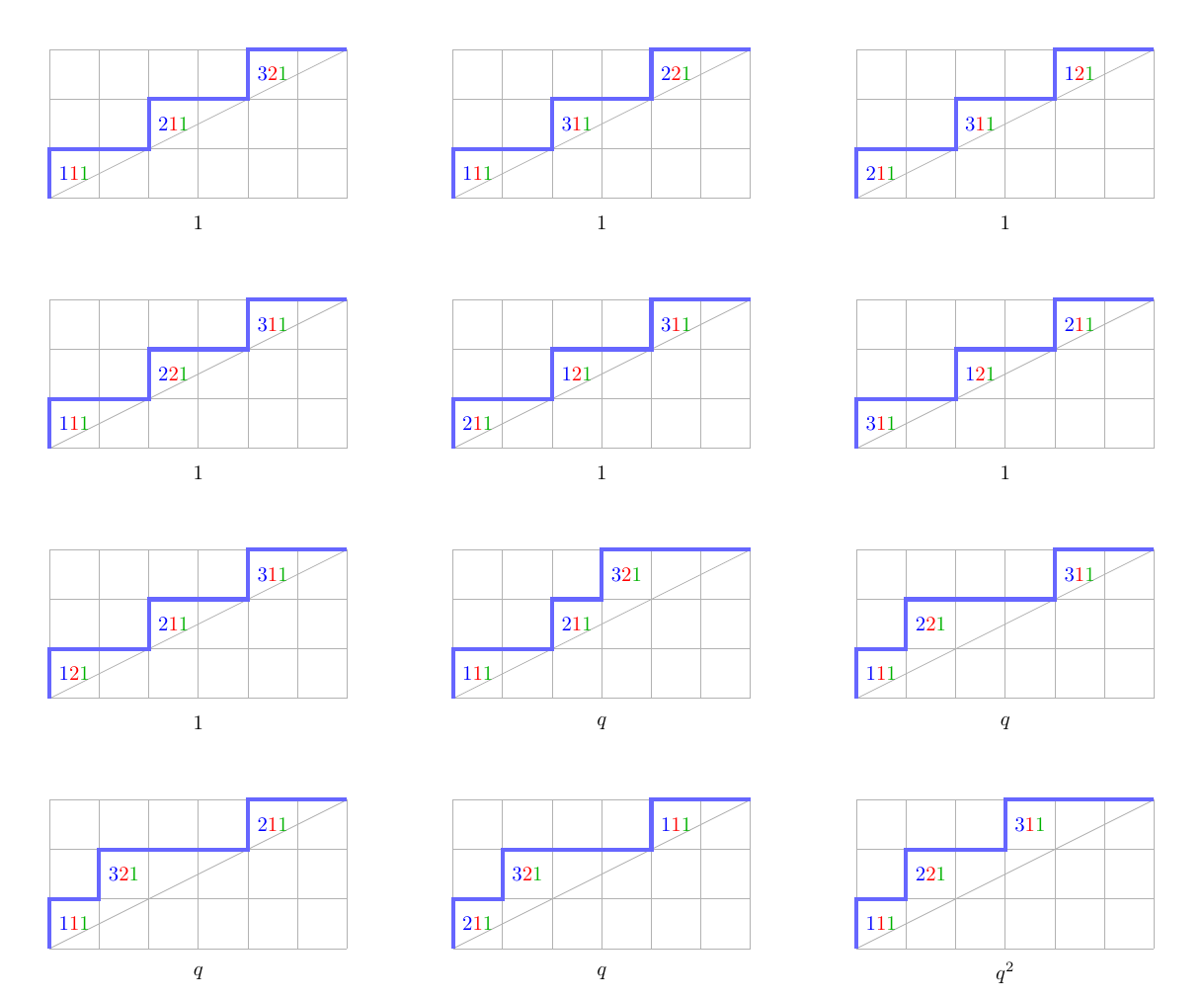}
	\end{center}
	\caption{The set of multilabeled Dyck paths in $\LD_{2^3}$ with $\bleu{\X}$-labels $(\bleu{1,2,3})$, $\rouge{\Y}$-labels $(\rouge{1,1,2})$, $\vertfonce{\Z}$-labels $(\vertfonce{1,1,1})$, giving the coefficient $q^2+4q+7$ of $m_{111}(\bleu{\X}) \otimes m_{21}(\rouge{\Y}) \otimes m_3(\vertfonce{\Z})$ in $\widehat{\nabla}_{\vertfonce{\Z}} \widehat{\nabla}_{\rouge{\Y}} (e_3(\bleu{\X}))$.}
	\label{figure:LD33_paths}
\end{figure}

\begin{tiny}
	\begin{table}
		\renewcommand{\arraystretch}{1.8}
		\begin{tabular}{|c||l|l|l|}
			\hline
			Coeff of $m_{111}(\bleu{\X})$ & $m_{111}(\vertfonce{\Z})$
			                              & $m_{21}(\vertfonce{\Z})$
			                              & $m_{3}(\vertfonce{\Z})$                                              \\
			\hline\hline
			$m_{111}(\rouge{\Y})$         & $\begin{aligned}     & q^{6} + 7 q^{5} + 26 q^{4} + 67 q^{3} \\[-2pt]
                    & \qquad + 134 q^{2} + 196 q + 163
				                                 \end{aligned}$
			                              & $q^{5} + 6 q^{4} + 20 q^{3} + 47 q^{2} + 78 q + 73$
			                              & $q^{3} + 5 q^{2} + 14 q + 19$                                        \\ \hline
			$m_{21}(\rouge{\Y})$           & $q^{5} + 6 q^{4} + 20 q^{3} + 47 q^{2} + 78 q + 73$
			                              & $q^{4} + 5 q^{3} + 15 q^{2} + 29 q + 31$
			                              & $q^{2} + 4 q + 7$                                                    \\ \hline
			$m_{3}(\rouge{\Y})$           & $q^{3} + 5 q^{2} + 14 q + 19$
			                              & $q^{2} + 4 q + 7$
			                              & $1$                                                                  \\ \hline
		\end{tabular}
		\bigskip

		\begin{tabular}{|c||l|l|l|}
			\hline
			Coeff of  $m_{21}(\bleu{\X}))$ & $m_{111}(\vertfonce{\Z})$
			                               & $m_{21}(\vertfonce{\Z})$
			                               & $m_{3}(\vertfonce{\Z})$                             \\
			\hline\hline
			$m_{111}(\rouge{\Y})$          & $q^{5} + 6 q^{4} + 20 q^{3} + 47 q^{2} + 78 q + 73$
			                               & $q^{4} + 5 q^{3} + 15 q^{2} + 29 q + 31$
			                               & $q^{2} + 4 q + 7$                                   \\ \hline
			$m_{21}(\rouge{\Y})$           & $q^{4} + 5 q^{3} + 15 q^{2} + 29 q + 31$
			                               & $q^{3} + 4 q^{2} + 10 q + 12$
			                               & $q + 2$                                             \\ \hline
			$m_{3}(\rouge{\Y})$            & $q^{2} + 4 q + 7$
			                               & $q + 2$
			                               & $0$                                                 \\ \hline
		\end{tabular}
		\bigskip

		\begin{tabular}{|c||l|l|l|}
			\hline
			Coeff of $m_{3}(\bleu{\X}))$ & $m_{111}(\vertfonce{\Z})$
			                             & $m_{21}(\vertfonce{\Z})$
			                             & $m_{3}(\vertfonce{\Z})$       \\
			\hline\hline
			$m_{111}(\rouge{\Y})$        & $q^{3} + 5 q^{2} + 14 q + 19$
			                             & $q^{2} + 4 q + 7$
			                             & $1$                           \\ \hline
			$m_{21}(\rouge{\Y})$         & $q^{2} + 4 q + 7$
			                             & $q + 2$
			                             & $0$                           \\ \hline
			$m_{3}(\rouge{\Y})$          & $1$
			                             & $0$
			                             & $0$                           \\ \hline
		\end{tabular}
		\caption{The $m$-expansion of $\widehat{\nabla}_{\vertfonce{\Z}} \widehat{\nabla}_{\rouge{\Y}}(e_3(\bleu{\X}))$.}
		\label{Table:nabla_Y_Z}
	\end{table}
\end{tiny}

\begin{theorem}
	\label{thm:monomial-expansion}
	For any $n$ and $k$,
	\begin{equation}
		\widehat{\nabla}_\ast^k e_n = \sum_{ (\pi,w) \in \LD_{k^n}} q^{\area(\pi)} \XX^w.
	\end{equation}
\end{theorem}

We want to construct an element in $\LD_{k^n}$ starting from a path $\pi$ and word $\overline{w} = (\overline{w}_1,\dots, \overline{w}_n)$ with $\overline{w}_i = (w_{i1} , \dots, w_{ik}) \in \mathbb{N}_+^k$ (corresponding to the last $k$ tensors). This means that $(\pi , \overline{w}) \in \LD_{k^n}^k$. We will now create $w = (w_1,\dots, w_n)$ with ${w}_{i} = ( w_{i0}, w_{i1} , \dots, w_{ik})$ so that the pair $L =(\pi,w)$ is a multi-labeled Dyck path. We will then also write $w(L) = w$ and $\overline{w}(L) = \overline{w}$.

Suppose there are $r$ weak descents in $\overline{w}$ at position $i$. Then $\pi$ must have at least $r$ east steps succeeding its $i^{\ith}$  north step. If there are precisely $r$ east steps directly after the $i^{\ith}$  north step, then we require $w_{i0} > w_{i+1,0}$.  If there are more than $r$ east steps instead, then we are free to choose any value for $w_{i+1,0}$. This gives an expansion as linear combination of products of elementary symmetric functions (see \Cref{equation:elementary_expansion} below).

\begin{definition}
	\label{def:e-composition}
	For a multilabeled Dyck path $L = (\pi, w)$,  we denote by $\eta(p)$ the unique composition of $n$ whose corresponding partial sums set is such that
	\[
		i \in \Sigma(\eta(L)) \qquad \text{ if and only if} \qquad  \# \{ j \mid w_{i,j} \geq w_{i+1,j} \} < \col_\pi(i+1) - \col_\pi(i).
	\]
\end{definition}

\begin{corollary}
	For any $n$ and $k$, we have
	\begin{equation}\label{equation:elementary_expansion}
		\widehat{\nabla}_\ast^k e_n = \sum_{L \in \LD_{k^n}^k} q^{\area(L)} e_{\eta(L)} \otimes \XX^{w(L)}.
	\end{equation}
\end{corollary}

The \emph{content} of a multi-labeled $k^n$-Dyck path $(\pi,w)$ is the multiplicity type $m(w)= (m(w_{\star 0}) ,\dots, m(w_{\star k}))$. We will then write $w \in W_\lambda$ if $m(w_{\star 0}) = 1^{\lambda_1} \cdots n^{\lambda_n}$. Denote by $\LD_{k^n}(\lambda)$ the set of $(\pi,w) \in \LD_{k^n}$ with $w \in W_\lambda$.

\begin{corollary}
	For any $\lambda \vdash n $ and $k$, we have
	\[\left. \nabla_\ast^k \Xi e_\lambda \right\rvert_{t=1} = \sum_{L \in \LD_{k^n}(\lambda)} q^{\area(L)} e_{\eta(L)} \otimes \XX^{\overline{w}(L)}. \]
\end{corollary}

This equality follows from the bijection in \Cref{def:iota}.

\subsection{From polyominoes to Dyck paths}
For $\lambda \vdash n$, let $W_\lambda \subseteq (\mathbb{N}^{k+1})^n$ be the set of words $w = w_1 \cdots w_n$ with $m(w_{\star0}) =  1^{\lambda_1} \cdots n^{\lambda_n}$ and $w_{i,j} > 0$ for $j > 0$.

In this section we complete the construction by giving a weight-preserving bijection between $\PP_{W_\lambda,\eta}$ and the set \[ \LD_{k^n}(\lambda,\eta)  \coloneqq \{ L \in \LD_{k^n}(\lambda) \mid \eta(L) = \eta \} \] of multi-labeled $k^n$-Dyck paths with content $\lambda$ and $e$-composition $\eta$. It is slightly easier to describe the inverse, so that is what we will do.

\begin{definition}
	\label{def:iota}
	Let $L = (\pi, w)$ be a multi-labeled $k^n$-Dyck path, let $P = \pi E$ and $Q = E (E^k N)^n$. We define $\iota(L) = (\overline{P}, \overline{Q}, w)$, where $\overline{P}$ and $\overline{Q}$ are the paths obtained from $P$ and $Q$ respectively by removing, for every $0 \leq i < n$, $\# \{ j \mid w_{ij} \geq w_{i+1,j} \}$ east steps of $P$ on the line $y=i$ and the same amount of east steps of $Q$ on the line $y=i-1$.
\end{definition}

\begin{proposition}
	The map $\iota$ defines a weight-preserving bijection between $\LD_{k^n}(\lambda, \eta)$ and $\PP_{W_\lambda,\eta}$, meaning $\area(L) = \area(\iota(L))$.
\end{proposition}

\begin{proof}[\bf Proof]
	First, notice that, since $\col_\pi(i+1) - \col_\pi(i) \geq \# \{ j \mid w_{ij} \geq w_{i+1,j} \}$, there must be at least $\# \{ j \mid w_{ij} \geq w_{i+1,j} \}$ east steps of $P$ on the line $y=i$. Notice that by definition we have at least $k$ east steps of $Q$ on each line. Since the first east step of $P$ on the line $y = i$ must necessarily be strictly left of the first east step of $Q$ along the line $y=i-1$, then $(\overline{P}, \overline{Q})$ is necessarily a parallelogram polyomino.

	Now, the number of east steps of $\overline{Q}$ on the line $y=i-1$ is exactly
	\[
		k - \# \{ j \mid w_{i,j} \geq w_{i+1,j} \} + \delta_{i,1} = \# \{ j \mid w_{i,j} < w_{i+1,j} \} + \delta_{i,1} = \lwt(w,i) + \delta_{i,1}.
	\]
	Thus $\iota(L)$ is $\rho$-compatible. By construction, the word of $\iota(L)$ is the same as the word of $L$, and so their contents are the same.
	By \Cref{def:e-composition}, the $e$-composition of $L$ is exactly given by the lenghts of the maximal vertical segments of $\overline{P}$. It follows that $\iota(L) \in \PP_{W_\lambda,\eta}$.

	Since the process that defines $\iota$ is invertible, the map is bijective. Finally, notice that the number of cells in the $i^{\ith}$ row decreases by $\# \{ j \mid w_{ij} \geq w_{i+1,j} \}$, but so does the number of east steps of the bottom path on the line $y=i-1$. This means that the area of $\iota(L)$ is the same as the area of $L$.
\end{proof}

\subsection{Schur expansions}
The Schur expansion of the modified Macdonald polynomials at $t=1$ is given by
\[ \frac{ \widehat{h}_\beta(\X)}{ \widehat{h}_\beta(1)} = \sum_{ \lambda \vdash n} s_\lambda(\X) \sum_{\alpha \in \LW_\lambda} q^{\revmaj_\beta}(\alpha), \]
where $\LW_\lambda$ is the set of lattice words of type $\lambda$. Given a word $w=(w_1,\dots,w_n)$ of tuples  of length $k$, we will say that $w$ is lattice of type $(\lambda^1,\dots, \lambda^{k})$ if $w_{\star i}$ is a lattice word of type $\lambda^i(w) = \lambda^i$.
We note that
\[ \widehat{\nabla}_\ast^k e_n = \sum_{\eta \vdash n} e_\eta \otimes s_{\lambda^1} \otimes \cdots s_{\lambda^k} D_{\LW_{\lambda^1,\dots,\lambda^k},\eta}(q), \]
where $\LW_{\lambda^1,\dots,\lambda^k}$ is the set of lattice words of type $(\lambda^1,\dots, \lambda^k)$, and the local weights we are using are the same as in the beginning of this section, given by the number of local ascents. Let $\LLD_{k^n}$ be the subset of $L \in \LD_{k^n}^k$ for which $w(L)$ is lattice. It follows immediately from the work developed in the previous sections that we have the following.

\begin{theorem}
	\label{thm:schur-expansion}
	\begin{equation}
		\widehat{\nabla}_\ast^k e_n  = \sum_{L \in \LLD_{k^n}} q^{\area(L)} e_\eta \otimes s_{\lambda^1(w(L))} \otimes \cdots \otimes s_{\lambda^k (w(L))}.
	\end{equation}
\end{theorem}

\subsection{Notable special cases}

\subsubsection{\texorpdfstring{$\lambda = (n), k = 1$}{lambda = (n), k = 1}}

When $k=1$, we get back regular Dyck paths, but with two distinct sets of labels.
By taking scalar products appropriately, we can get back some well-known combinatorial expressions, such as the shuffle theorem or the valley version of the Delta conjecture.

Indeed, by \Cref{prop:scalar_product}, we have $\langle e_n, \nabla_\ast \rangle e_n = \nabla e_n$. This corresponds to the case in which the $\Y$-labels appear in strictly increasing order in the reading word. Since we are working with the area only, which does not depend on the labels, we can use either the diagonal reading word (in accordance with the $(\dinv, \area)$ version of the shuffle theorem) or the vertical reading word (in accordance with the equivalent $(\area, \pmaj)$ version \cite{Loehr-Remmel-2004}).

In either case, the fact that the $\Y$-labels must be strictly increasing in the reading order implies that, whenever two north steps are separated by one single east step, their label strictly increase, and since $r=2$, there is no restriction on the $\X$-labels other than the fact that they must increase along columns, which is exactly the condition appearing in the shuffle theorem. 

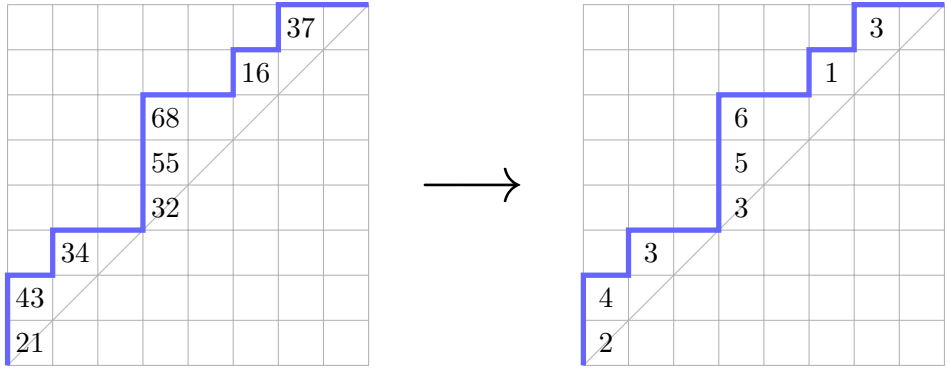
\begin{figure}[!ht]
	\centering
	\begin{tikzpicture}[scale=.6]
		\draw[step=1.0, gray!60, thin] (0,0) grid (8,8);

		\begin{scope}
			\clip (0,0) rectangle (8,8);
			\draw[gray!60, thin] (0,0) -- (8,8);
		\end{scope}

		\draw[blue!60, line width=2pt] (0,0) -- (0,1) -- (0,2) -- (1,2) -- (1,3) -- (2,3) -- (3,3) -- (3,4) -- (3,5) -- (3,6) -- (4,6) -- (5,6) -- (5,7) -- (6,7) -- (6,8) -- (7,8) -- (8,8);

		\node at (0.5,0.5) {$21$};
		\node at (0.5,1.5) {$43$};
		\node at (1.5,2.5) {$34$};
		\node at (3.5,3.5) {$32$};
		\node at (3.5,4.5) {$55$};
		\node at (3.5,5.5) {$68$};
		\node at (5.5,6.5) {$16$};
		\node at (6.5,7.5) {$37$};
	\end{tikzpicture}%
	\begin{tikzpicture}[scale=.6]
		\node at (-2,0) {};
		\node at (0,3.75) {\Huge $\longrightarrow$};
		\node at (2,0) {};
	\end{tikzpicture}
	\begin{tikzpicture}[scale=.6]
		\draw[step=1.0, gray!60, thin] (0,0) grid (8,8);

		\begin{scope}
			\clip (0,0) rectangle (8,8);
			\draw[gray!60, thin] (0,0) -- (8,8);
		\end{scope}

		\draw[blue!60, line width=2pt] (0,0) -- (0,1) -- (0,2) -- (1,2) -- (1,3) -- (2,3) -- (3,3) -- (3,4) -- (3,5) -- (3,6) -- (4,6) -- (5,6) -- (5,7) -- (6,7) -- (6,8) -- (7,8) -- (8,8);

		\node at (0.5,0.5) {$2$};
		\node at (0.5,1.5) {$4$};
		\node at (1.5,2.5) {$3$};
		\node at (3.5,3.5) {$3$};
		\node at (3.5,4.5) {$5$};
		\node at (3.5,5.5) {$6$};
		\node at (5.5,6.5) {$1$};
		\node at (6.5,7.5) {$3$};
	\end{tikzpicture}
	\caption{A two-labeled Dyck path with $\Y$-(diagonal) reading word $12345678$, and the corresponding labeled Dyck path.}
	\label{fig:catalan}
\end{figure}

If we take instead $\langle e_{n-k} h_k, \nabla_\ast \rangle e_n = \Delta_{e_{n-k}} e_n$, this corresponds to the case in which the reading word of the $\Y$-labels is a shuffle of $(1,2,\dots,n-k)$ and $(n, n-1, \dots, n-k+1)$. As usual in this case, we can interpret this as decorating $k$ peaks of the path, with the condition that whenever there is a decorated peak and a  north step in the column immediately to its right, since the $\Y$-labels decrease, the $\X$-labels must increase. This condition is exactly the same as saying that the valley following a decorated peak must be contractible, which is the condition given by the valley version of the Delta conjecture. Of course the topmost peak does not correspond to any valley, and in fact we get $\Delta_{e_{n-k}} e_n$ rather than $\Delta'_{e_{n-k-1}} e_n$.

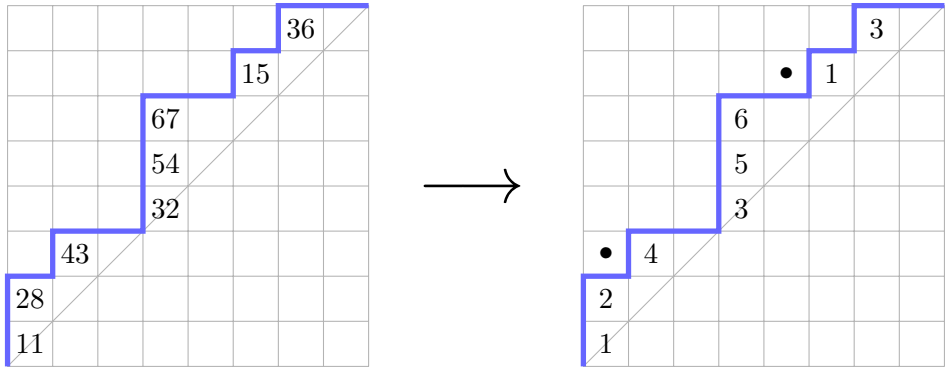
\begin{figure}[!ht]
	\centering
	\begin{tikzpicture}[scale=.6]
		\draw[step=1.0, gray!60, thin] (0,0) grid (8,8);

		\begin{scope}
			\clip (0,0) rectangle (8,8);
			\draw[gray!60, thin] (0,0) -- (8,8);
		\end{scope}

		\draw[blue!60, line width=2pt] (0,0) -- (0,1) -- (0,2) -- (1,2) -- (1,3) -- (2,3) -- (3,3) -- (3,4) -- (3,5) -- (3,6) -- (4,6) -- (5,6) -- (5,7) -- (6,7) -- (6,8) -- (7,8) -- (8,8);

		\node at (0.5,0.5) {$11$};
		\node at (0.5,1.5) {$28$};
		\node at (1.5,2.5) {$43$};
		\node at (3.5,3.5) {$32$};
		\node at (3.5,4.5) {$54$};
		\node at (3.5,5.5) {$67$};
		\node at (5.5,6.5) {$15$};
		\node at (6.5,7.5) {$36$};
	\end{tikzpicture}%
	\begin{tikzpicture}[scale=.6]
		\node at (-2,0) {};
		\node at (0,3.75) {\Huge $\longrightarrow$};
		\node at (2,0) {};
	\end{tikzpicture}
	\begin{tikzpicture}[scale=.6]
		\draw[step=1.0, gray!60, thin] (0,0) grid (8,8);

		\begin{scope}
			\clip (0,0) rectangle (8,8);
			\draw[gray!60, thin] (0,0) -- (8,8);
		\end{scope}

		\draw[blue!60, line width=2pt] (0,0) -- (0,1) -- (0,2) -- (1,2) -- (1,3) -- (2,3) -- (3,3) -- (3,4) -- (3,5) -- (3,6) -- (4,6) -- (5,6) -- (5,7) -- (6,7) -- (6,8) -- (7,8) -- (8,8);

		\node at (0.5,0.5) {$1$};
		\node at (0.5,1.5) {$2$};
		\node at (1.5,2.5) {$4$};
		\node at (3.5,3.5) {$3$};
		\node at (3.5,4.5) {$5$};
		\node at (3.5,5.5) {$6$};
		\node at (5.5,6.5) {$1$};
		\node at (6.5,7.5) {$3$};
		\draw (0.5,2.5) node {$\bullet$};
		\draw (4.5,6.5) node {$\bullet$};
	\end{tikzpicture}
	\caption{A two-labeled Dyck path with $\Y$-(diagonal) reading word $12834567 \in 123456 \shuffle 87$ and the corresponding valley-decorated labeled Dyck path. The ``big cars'' are replaced by decorated valleys.}
	\label{fig:labelled-dyck-path}
\end{figure}

In both cases, the valley version of the Delta conjecture from \cite{Haglund-Remmel-Wilson-2018} gives us a $t$-statistic on our two-labeled Dyck paths with certain $\Y$-labelings. We hope to extend this statistic to the full set of multi-labeled Dyck paths and any labeling.


\subsection{Tiered Trees}
It is worth noticing that, thanks to the following identity, we have another interpretation for $\nabla_\ast e_n$ that extends a conjecture of \cite{DAdderio-Iraci-LeBorgne-Romero-VandenWyngaerd-2022}.

\begin{proposition}
	\label{Equation:Nabla_Xi}
	\[ \nabla_\ast e_n = \sum_\lambda m_\lambda \otimes (\Xi e_\lambda) \]
\end{proposition}

Combining this with \cite{DAdderio-Iraci-LeBorgne-Romero-VandenWyngaerd-2022}*{Conjecture~6.4}, we have an $(m \otimes m)$-expansion for $\left. \nabla_\ast e_n \right\rvert_{t=1}$ in terms of rooted tiered trees and tree inversions.

Recall that $\Xi e_{1^n} = \Delta_{e_1} M \Pi (e_{1^n}^\ast)$. The coefficient of $m_\lambda$ in $M \Pi (e_{1^n}^\ast)$ is the Kac polynomial of certain dandelion quivers with a dimension vector depending on $\lambda$; which is to say that it is the polynomial enumerating absolutely indecomposable representations of the quiver over $\mathbb{F}_q$ \cites{DAdderio-Iraci-LeBorgne-Romero-VandenWyngaerd-2022,Gunnells-Letellier-Rodriguez-Quivers-2018}. It would  be interesting to find a similar interpretation for any $\Xi e_\mu$, not just for $\mu = (1^n)$.

\section{On power sums}
\label{section:power}

Our goal in this section is to show the following.

\begin{proposition}
	For any nonempty partition $\gamma$,
	\begin{align*}
		\widehat{\nabla}_\ast \widehat{\Delta}_{m_\gamma} (\omega p_n)
			& = \sum_{L \in \RLD^\gamma_{k^n}} q^{\area(L)} e_{\eta(L)} \otimes \XX^L \\
			& = \sum_{L \in \RLLD^\gamma_{k^n}} q^{\area(L)} e_{\eta(L)} \otimes s_{\lambda^1(w(L))} \otimes \cdots \otimes s_{\lambda^k (w(L))},
	\end{align*}
	where the sums are respectively over multi-labeled $(\gamma,k^n)$-Dyck paths with a marked return, and \underline{lattice} multi-labled $(\gamma,k^n)$-Dyck paths with a marked return.
\end{proposition}

Definitions for the relevant sets are given in the subsequent sections, and they come directly from the general framework of \Cref{section:weights}.

\subsection{Newton's Identity}
\label{sec_Newton}
The set of \emph{selected rearrangements} $\SR(\mu)$ of a partition $\mu$ are all possible pictures that may be obtained by the following procedure:
\begin{enumerate}
	\item select a rearrangement $(\alpha_1, \dots, \alpha_{\ell(\mu)})$ of the parts of $\mu$, drawn as a list of rows;
	\item select one of the cells in the first part $\alpha_1$, and mark the cell with a circle.
\end{enumerate}
For instance, to get an element of $\SR( 4,3,1,1 )$, we may rearrange the parts as $(3,1,4,1)$, and one of its $3$ possible markings is
\ywhite
\begin{align*}
	\gyoung(;;{\Circle};)\quad \yng(1)\quad \yng(4)\quad  \yng(1).
\end{align*}
It follows by the construction of this set that one has
\begin{equation}
	\lvert \SR(\mu) \rvert = \sum_{\alpha \in \R(\mu)} \alpha_1.
\end{equation}

\begin{proposition}[Newton's Identity]
	For any $n$, \[	p_n  = \sum_{\mu \vdash n} (-1)^{\ell(\mu)-1 } \lvert \SR(\mu) \rvert \, h_{\mu}. \]
\end{proposition}

It follows that \[ \widehat{p}_n(\X) = \frac{p_n(\X)}{1-q^n} = - \sum_{\mu \vdash n} (-1)^{\ell(\mu) } \lvert \SR(\mu) \rvert \, \widehat{h}_{\mu}, \]
or equivalently
\begin{equation}
	\label{Equation:omegapn}
	\omega(p_n) = (1-q^n) \sum_{\mu \vdash n} (-1)^{n-\ell(\mu)} \lvert \SR(\mu) \rvert \, \widehat{h}_{\mu}.
\end{equation}

Then we have the following.
\begin{proposition}
	For any $n$, $k$, and partition $\gamma$,
	\begin{align*}
		\widehat{\nabla}_\ast^k \widehat{\Delta}_{m_\gamma} \omega(p_n) = \sum_{\eta \vdash n} e_\eta(\X_0) \sum_{\mu \vdash n} & \sum_{\vec{\nu} \in \PR(\mu,\eta)} (-1)^{n-\ell(\mu)} (1-q^n) \bC_{\nu^1} \cdots \bC_{\nu^{\ell(\mu)}} m_{\gamma}\Big[{\textstyle \sum_{i} [\mu_i]_q }\Big] \\
		& \times \sum_{\alpha \in \R(\mu)} \alpha_1 \sum_{w \in (\mathbb{N}_{+}^k)^n} \prod_{i=1}^k q^{\revmaj_\alpha(w_{\star i})} x_{i ,w_{1i}} \cdots x_{i, w_{ni}}.
	\end{align*}
\end{proposition}
\begin{proof}
	Applying $\nabla_\ast^k$ to \Cref{Equation:omegapn}, we get
	\begin{align*}
		\widehat{\nabla}_\ast^k \widehat{\Delta}_{m_\gamma} \omega(p_n) = \sum_{\mu \vdash n} (-1)^{n-\ell(\mu)} (1-q^n) \lvert \SR(\mu) \rvert \, \left( \frac{\widehat{h}_{\mu}}{\widehat{h}_{\mu}(1)} \right)^{\otimes k} \otimes  m_{\gamma} \Big[{\textstyle \sum_{i} [\mu_i]_q } \Big]\, \widehat{h}_{\mu}.
	\end{align*}
	Recall that, using the Cauchy identity, we have \[ \widehat{h}_\mu = \sum_{\eta \vdash n} e_\eta \sum_{\vec{\nu} \in \PR (\eta,\mu)} f_{\vec{\nu}} \left[ {1}/{(1-q)} \right]; \]
	we may now interpret the product
	\[ (1-q^n)  f_{\vec{\nu}} \left[ {1}/{(1-q)} \right] = (1-q^n) (-1)^{n-\ell(\eta)} \bC_{\nu^1} \cdots \bC_{\nu^{\ell(\mu)}} \]
	as the generating function of column-composition tableaux $(C^1,\dots, C^{\ell(\mu)})$ with $C^i \in \CC_{\nu^i}$ such that $c_1(C^i) = 0$ for at least one of the $C^i$. The monomial expansion of the product $\widehat{h}_\alpha ^{\otimes k}$ is given by a sum over all words $w = (w_1,\dots, w_n)$ with $w_i \in (\mathbb{N_+})^k$, where if $w_i = (w_{i1},\dots, w_{in}) $, then $w$ contributes the term \[ \prod_{i=1}^{k} q^{\revmaj_\alpha(w_{\star i})} x_{i ,w_{1i}} \cdots x_{i, w_{ni}}. \qedhere \]%
\end{proof}

Each of the components in the summand can thus be given a combinatorial interpretation, and we are led to the following definition.

\begin{definition} \label{SCTDefinition}
	Let $\eta,\mu \vdash n$, be partitions of $n$, and let $\gamma$ be any nonempty partition. We define a sequence of \emph{labeled selected column-composition tableaux} of type $(\eta,\mu,\gamma)$, denoted by $S = (S^1,\dots, S^\ell) \in \SC^\gamma(\eta, \mu)$, as the result of the following process.
	\begin{enumerate}
		\item Select $\alpha \in \SR(\mu)$.
		\item Select $(\nu^1,\dots, \nu^{\ell(\alpha)}) \in \PR({\eta,\alpha} )$.
		\item For $i=1,\dots, \ell(\alpha)$ select $C^i \in {\CC_{\nu^i}}$ with the condition that for some $j$, $c_1 (C^j) = 0$. The selection in $\alpha_1$ is represented by placing the circle above its corresponding column in $C^1$.
		\item For notation, the words will be indexed by $j=0,\dots, k$; the $j^{\ith}$ word in the sequence $w^i$ is denoted by $w^{i}_j$; the $r^{\ith}$ letter in the word $w^i_j$ will be denoted by $w^i_{jr}$; and we denote the word $w^i_{1r}w^{i}_{2r} \cdots w^{i}_{\alpha_i,r}$ by $w^i_{\star r}$.

		For each $i = 1,\dots, \ell(\alpha)$ select a sequence of words $w^i\in (\mathbb{N}^{k+1})^{\alpha_i}$ with two extra conditions: the first condition is that $w^i_{jr}$ is nonzero for $r>0$; the second condition is that $m(w_{\star 0}) = 0^{n-\ell(\gamma)}1^{m_1(\gamma_1)} \cdots n^{m_n(\gamma_n)}$, so that, in other words, $w_{\star 0} \in \R(0^{n-\ell(\gamma)},\gamma).$
		\item Set $S^i = (C^i, w^i)$.
	\end{enumerate}
	
	Now we define \[ \rho(S^i) = q^{\lvert C^i \rvert} q^{\rw(w^{i}_{\ast 0})} \prod_{j=1}^k q^{\revmaj(w^{i}_{\star j})} x_{j,w^i_{1j}} \cdots x_{j,w^i_{\ell(w^i),j}} \] with \[ \rw(a_1,\dots, a_\ell) = \sum_{i=1}^{\ell} (\ell-i) a_i, \] and set $\rho(S) = \rho(S^1) \cdots \rho(S^{\ell(\eta)})$. We finally define the sign of $S$ as \[ \sign(S) = (-1)^{\ell(\eta)- \ell(\mu)}, \] so that it	gives again the parity for the number of vertical bars appearing in $S$.
\end{definition}
With the notation of the previous section, the statistic which looks right in this case is given by the reverse major index (and the product of the appropriate $x_{ij}$ monomials) along the last $k$ entries in the tuples of $\vec{w} = (w^1,\dots, w^{\ell(\alpha)})$, and by $\rw$ in the first entries of all the tuples.

For instance, suppose $\eta = (3,2^4,1^5)$, $\mu = (5,4,3^2,1)$, and $\gamma = (2,2,1,1)$, and consider the rearrangement $\alpha = (5,3,4,1,3)$.  We select
\[
	\vec{\nu} = ((2,2,1),(2,1),(2,1,1),(1),(3)) \in \PR({\eta,\mu}),
\]
and choose the following column-composition tableaux in $CC_{\nu^1},\dots, CC_{\nu^{\ell(\rho)}}$.
\begin{align*}
	\begin{tikzpicture}[xscale=-1,scale=.5]
		\Yboxdim{1cm}
		\ygreen
		\tyng(0cm,0cm,5,2)
		\ylw
		\tyng(0cm,-1cm,5)
		\draw[line width = .5mm] (-.25,1-1) -- (5.25,1-1);
		\draw[line width = .5mm] (2,-1) -- (2,2.25 );
		\draw[line width = .5mm] (4,-1) -- (4,1.25 );
	\end{tikzpicture} &  &
	\begin{tikzpicture}[xscale=-1,scale=.5]
		\Yboxdim{1cm}
		\ylw
		\tyng(0cm,-1cm,3)
		\draw[line width = .5mm] (-.25,1-1) -- (3.25,1-1);
		\draw[line width = .5mm] (1, -1) -- (1,1.25-1);
	\end{tikzpicture}
	                                                   &  &
	\begin{tikzpicture}[xscale=-1,scale=.5]
		\Yboxdim{1cm}
		\ygreen
		\tyng(0cm,0cm,4,2)
		\ylw
		\tyng(0cm,-1cm,4)
		\draw[line width = .5mm] (-.25,1-1) -- (4.25,1-1);
		\draw[line width = .5mm] (1, -1) -- (1,3.25-1);
		\draw[line width = .5mm] (2, -1) -- (2,3.25-1);
	\end{tikzpicture}
	                                                   &  &
	\begin{tikzpicture}[xscale=-1,scale=.5]
		\Yboxdim{1cm}
		\ygreen
		\tyng(0cm,0cm,1)
		\ylw
		\tyng(0cm,-1cm,1)
		\draw[line width = .5mm] (-.25,0) -- (1.25,0);
	\end{tikzpicture}
	                                                   &  &
	\begin{tikzpicture}[xscale=-1,scale=.5]
		\Yboxdim{1cm}
		\ylw
		\tyng(0cm,-1cm,3)
		\draw[line width = .5mm] (-.25,1-1) -- (3.25,1-1);
	\end{tikzpicture}
\end{align*}
Note here that at least one column-composition tableau has no cells in the first column. Now place a circle above a column in the first component, and replace each cell in the base row with a column of $k+1$ numbers. From left to right, the first column has the tuple $w^1_1$, the second column has the tuple $w^1_2$, and so on until we reach column $\alpha_1$ containing the tuple $w^1_{\alpha_1}$. The next column is in the second column composition tableau, and it contains the tuple $w^2_1$. We continue in this way to create $S'$ as in \Cref{figure:s_prime}.

\begin{figure}[ht]
	\begin{align*}
		\begin{tikzpicture}[ xscale=-1]
			\draw (.5,1) node {$S' =$};
			\draw (0,0) circle (0cm);
		\end{tikzpicture}                    &  &
		\begin{tikzpicture}[xscale=-1,scale=.5]
			\Yboxdim{1cm}
			\ygreen
			\tyng(0cm,0cm,5,2)
			\ylw
			\tyng(0cm,-3cm,5,5,5)
			\draw[line width = .5mm] (-.25,1-1) -- (5.25,1-1);
			\draw[line width = .5mm] (2,-3) -- (2,2.25 );
			\draw[line width = .5mm] (4,-3) -- (4,1.25 );
			\begin{large}
				\draw (.5,.5-1-1) node {$1$};
				\draw (1.5,.5-1-1) node {$3$};
				\draw (2.5,.5-1-1) node {$7$};
				\draw (3.5,.5-1-1) node {$2$};
				\draw (4.5,.5-1-1) node {$1$};
				\draw (.5,.5-1-2) node {$3$};
				\draw (1.5,.5-1-2) node {$1$};
				\draw (2.5,.5-1-2) node {$8$};
				\draw (3.5,.5-1-2) node {$2$};
				\draw (4.5,.5-1-2) node {$1$};
				\draw (.5,.5-1) node {$0$};
				\draw (1.5,.5-1) node {$0$};
				\draw (2.5,.5-1) node {$0$};
				\draw (3.5,.5-1) node {$1$};
				\draw (4.5,.5-1) node {$0$};
				\draw (2.5,1.5) node {{\Circle}};
			\end{large}
		\end{tikzpicture} &  &
		\begin{tikzpicture}[xscale=-1,scale=.5]
			\Yboxdim{1cm}
			\ylw
			\tyng(0cm,-3cm,3,3,3)
			\draw[line width = .5mm] (-.25,1-1) -- (3.25,1-1);
			\draw[line width = .5mm] (1, -3) -- (1,1.25-1);
			\begin{large}
				\draw (.5,.5-1-1) node {$2$};
				\draw (1.5,.5-1-1) node {$4 $};
				\draw (2.5,.5-1-1) node {$2 $};
				\draw (.5,.5-1-2) node {$1$};
				\draw (1.5,.5-1-2) node {$5$};
				\draw (2.5,.5-1-2) node {$4 $};
				\draw (.5,.5-1) node {$2$};
				\draw (1.5,.5-1) node {$0$};
				\draw (2.5,.5-1) node {$0 $};
			\end{large}
		\end{tikzpicture}
		                                                   &  &
		\begin{tikzpicture}[xscale=-1,scale=.5]
			\Yboxdim{1cm}
			\ygreen
			\tyng(0cm,0cm,4,2)
			\ylw
			\tyng(0cm,-3cm,4,4,4)
			\begin{large}
				\draw[line width = .5mm] (-.25,1-1) -- (4.25,1-1);
				\draw[line width = .5mm] (1, -3) -- (1,3.25-1);
				\draw[line width = .5mm] (2, -3) -- (2,3.25-1);
				\draw (.5,.5-1-1) node {$3$};
				\draw (1.5,.5-1-1) node {$1$};
				\draw (2.5,.5-1-1) node {$3 $};
				\draw (3.5,.5-1-1) node {$9 $};
				\draw (.5,.5-1-2) node {$8$};
				\draw (1.5,.5-1-2) node {$2 $};
				\draw (2.5,.5-1-2) node {$3 $};
				\draw (3.5,.5-1-2) node {$4$};
				\draw (.5,.5-1) node {$0$};
				\draw (1.5,.5-1) node {$0 $};
				\draw (2.5,.5-1) node {$2 $};
				\draw (3.5,.5-1) node {$0$};
			\end{large}
		\end{tikzpicture}            &  &
		\begin{tikzpicture}[xscale=-1,scale=.5]
			\Yboxdim{1cm}
			\ygreen
			\tyng(0cm,0cm,1)
			\ylw
			\tyng(0cm,-3cm,1,1,1)
			\begin{large}
				\draw[line width = .5mm] (-.25,1-1) -- (1.25,1-1);
				\draw (.5,.5-1-1) node {$4$};
				\draw (.5,.5-1-2) node {$1$};
				\draw (.5,.5-1) node {$0$};
			\end{large}
		\end{tikzpicture}            &  &
		\begin{tikzpicture}[xscale=-1,scale=.5]
			\Yboxdim{1cm}
			\ylw
			\tyng(0cm,-3cm,3,3,3)
			\draw[line width = .5mm] (-.25,1-1) -- (3.25,1-1);
			\begin{large}
				\draw (.5,.5-1-1) node {$4$};
				\draw (1.5,.5-1-1) node {$1 $};
				\draw (2.5,.5-1-1) node {$1 $};
				\draw (.5,.5-1-2) node {$6$};
				\draw (1.5,.5-1-2) node {$6 $};
				\draw (2.5,.5-1-2) node {$1 $};
				\draw (.5,.5-1) node {$0$};
				\draw (1.5,.5-1) node {$1$};
				\draw (2.5,.5-1) node {$0$};
			\end{large}
		\end{tikzpicture}
	\end{align*}
	\caption{An example of $S' \in \SC^{(2^2,1^2)}((3,2^4,1^5),(5,4,3^2,1)) $.}
	\label{figure:s_prime}
\end{figure}
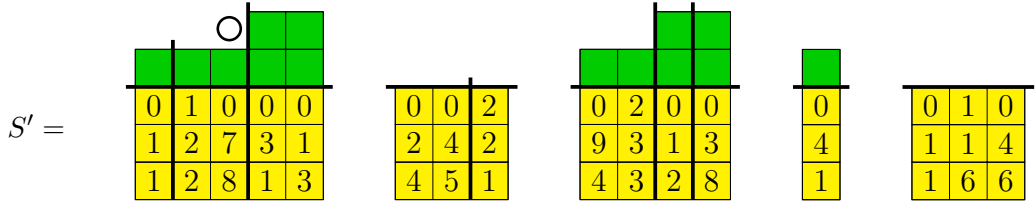

This is a selected, labeled column-composition tableau. We see that the first row of labels from the top (which is $w_{\ast 0}$) has nonzero entries which rearrange to $\gamma$. The second row of entries is given by $w_{\star 1} = w^1_{\star 1}\cdots w^{\ell(\alpha)}_{\star 1} = (1,2,7,3,1,2,4,2,9,3,1,3,4,1,1,4)$.

To compute the weight for this sequence we first count the number of cells above the base. This gives $q^{14}$. We now calculate the reverse major index for each row after the $0^{\ith}$ row. If we read the second row of labels in the first component, we get the word $w^1_{\star 1} = (1,2,7,3,1)$, which has an ascent in the first and second positions. The number of cells to the right of the first ascent is $4$, and the number of cells to the right of the second ascent is $3$. This gives $q^7$. Similarly, the last row in the first component gives the word $w^1_{\star 2} = (1,2,8,1,3)$, which has ascents in positions $1,2,4$. This gives $q^8$. Continuing in this fashion, the second labeled tableau gives $q^{2+2}$; the third gives $q^{1+1}$; and the last set of labels gives $q^{1+2}$. Therefore, all the labels corresponding to $w_{\star 1}$ and $w_{\star 2}$ contribute a total of $q^{24}$. If the label $j$ appears in $i^{\text{th}}$ row of labels, it contributes to the weight by $x_{ ij}$. Thus, the second row of labels in the first component of the sequence contributes to the total weight by a factor of $x_{11} x_{12} x_{17} x_{13} x_{11}$.

The $0^{\ith}$ row's weight is calculated as follows: In the first component, there is a $1$ in the second column. Since there are $3$ columns to its right, this label contributes a weight of $q^{1(3)}$. The second component has a two in the last column, which contributes a factor of $q^{2(0)}$. The third component gives $q^{2(2)}$, and the last component gives $q^{1(1)}$. The first row therefore contributes a factor of $q^{8}$.
In total, we have found that
\[
	\rho(S') = q^{14} q^8 q^{24} \, x_{11}^4 \, x_{12}^2 \, x_{13}^3 \, x_{14}^3 \, x_{17} \, x_{19} \, x_{21}^5 \, x_{22}^2 \, x_{23}^2 \, x_{24}^2 \, x_{25} \, x_{26}^2 \, x_{28}^2
\]

\begin{proposition} For any $k$, $n$, and nonempty partition $\gamma$, we have
	\[ \widehat{\nabla}_\ast^k \, \widehat{\Delta}_{m_\gamma} \omega(p_n) = \sum_{\eta \vdash n} \sum_{\mu \vdash n} \sum_{S \in \SC^\gamma(\eta,\mu)} \sign(S) \rho(S) \otimes e_\eta. \]
\end{proposition}

\subsection{A sign-reversing involution}

We define a variant of the split map from \cite{Romero-Thesis},\cite{IraciRomero2022DeltaTheta},\cite{Hicks-Romero-2018} that reduces the infinite, signed sum of selected, labeled column-composition tableaux to a finite number of fixed points. We start by defining the split map at at any vertical bar. To this end, we define the relative local weight of two sequences of words with equal lengths, as is done in \Cref{section:weights}.

Specializing the previous definitions, we have the following relative statistics.
	
\begin{definition}
	Let $A \in (\mathbb{N}^{k+1})^a$ and $B \in (\mathbb{N}^{k+1})^b$ be two words of tuples of length $k+1$. Then the \emph{relative statistic} between $A$ and $B$ is the quantity
	\[ \D(A;B) = A_{10}+ \cdots + A_{a0}  +\sum_{i=1}^k \asc(A_{1i}\cdots A_{ai} B_{1i}). \]
\end{definition}

\begin{definition}
	Let $T = (C,w )$ be a selected, labeled column-composition tableau, let $A = (w_1,\dots, w_r)$ be the the first $r$ letters of $w$, and let $B = (w_{r+1},\dots, w_\ell)$ be the remaining portion of $w$. Suppose there is a vertical bar in $T$ after column $r$. Then we define the \emph{split of $T$ at $r$} to be $\overline{\spl}_r(T)  = ( T^1,T^2)$ by the following procedure.
	\begin{enumerate}
		\item  Set
		      $T^1 = (C^1,A)$ where $C^1$ consists of the first $r$ columns of $C$. If one of the first $r$ columns of $T$ was selected, we keep the selection over the corresponding column in $C^1$.
		\item Set
		      $T^2 = (C^2,B)$ where $C^2$ is made by starting with the last $\ell-r$ columns of $C$ and adding
		      $
			      \D(A;B)
		      $
		      cells to each column.
		      If there was a selected column in $T$ in the last $\ell-r$ columns, then we keep the selection over the corresponding column in $C^2$.
	\end{enumerate}
\end{definition}

\begin{definition}
	Suppose $T^1 = (C^1,A)$ and $T^2 = (C^2,B)$ are the split of some selected column-composition tableau, $T$. This means that $\overline{\spl}_r(T) = (T^1,T^2)$ for some $r$. Then we say that $T^1$ and $T^2$ can \emph{join} and we set $\join(T^1,T^2) = T$. Such a $T$ is unique and exists if and only if the following inequality holds:
	\[
		c_1(T^2)  \geq c_\ell(T^1) +\D(A;B).
	\]
\end{definition}

Our previous work in \Cref{section:weights} implies that the map $\overline{\spl}_r$ is weight-preserving.

\begin{definition} \label{psicirc} Given $S = (S^1,\dots, S^{\ell}) \in \SC^\gamma(\eta,\mu)$, we say that the circle is \emph{leading} in $S$ if there are no vertical bars to its left (so that it is in the first component of $S^1$). Define a map $\psi_{\circ}$ by the following process:
	\begin{enumerate}
		\item If the circle is not leading in $S$, then there is a right-most vertical bar to the left of the circle, say occurring after column $r$. Then if $\overline{\spl}_r (S^1) = (R^1,R^2)$, set $\psi_{\circ}(S) = (R^2,S^2,\dots, S^\ell, R^1).$
		\item If the circle is leading and $S^\ell$ can join $S^1$, then $\psi_{\circ}(S) = ( \join (S^\ell,S^1), S^2,\dots, S^{\ell-1})$.
		\item If the circle is leading and $S^\ell$ cannot join $S^1$, then $\psi_{\circ}(S) = \psi(S)$, as given in \Cref{Definition:split}.
	\end{enumerate}
\end{definition}

For example, the configuration $S'$ of \Cref{figure:s_prime} maps to the configuration $\psi_{\circ}(S')$ in \Cref{figure:phis_prime}.

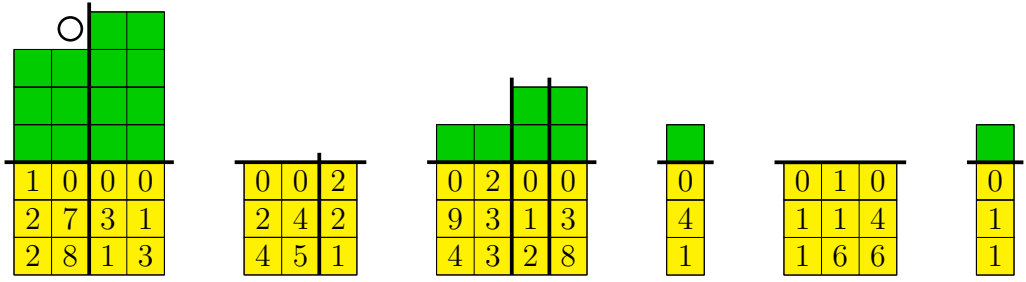
\begin{figure}[ht]
\begin{align*}
	\begin{tikzpicture}[xscale=-1,scale=.5]
		\Yboxdim{1cm}
		\ygreen
		\tyng(0cm,0cm,4,4,4,2)
		\ylw
		\tyng(0cm,-3cm,4,4,4)
		\draw[line width = .5mm] (-.25,1-1) -- (4.25,1-1);
		\draw[line width = .5mm] (2,-3) -- (2,4.25 );
		\begin{large}
			\draw (.5,.5-1-1) node {$1$};
			\draw (1.5,.5-1-1) node {$3 $};
			\draw (2.5,.5-1-1) node {$7 $};
			\draw (3.5,.5-1-1) node {$2 $};
			\draw (.5,.5-1-2) node {$3$};
			\draw (1.5,.5-1-2) node {$1$};
			\draw (2.5,.5-1-2) node {$8 $};
			\draw (3.5,.5-1-2) node {$2 $};
			\draw (.5,.5-1) node {$0$};
			\draw (1.5,.5-1) node {$0$};
			\draw (2.5,.5-1) node {$0 $};
			\draw (3.5,.5-1) node {$1 $};
			\draw (2.5,3.5) node {{\Circle}};
		\end{large}
	\end{tikzpicture} &  &
	\begin{tikzpicture}[xscale=-1,scale=.5]
		\Yboxdim{1cm}
		\ylw
		\tyng(0cm,-3cm,3,3,3)
		\draw[line width = .5mm] (-.25,1-1) -- (3.25,1-1);
		\draw[line width = .5mm] (1, -3) -- (1,1.25-1);
		\begin{large}
			\draw (.5,.5-1-1) node {$2$};
			\draw (1.5,.5-1-1) node {$4 $};
			\draw (2.5,.5-1-1) node {$2 $};
			\draw (.5,.5-1-2) node {$1$};
			\draw (1.5,.5-1-2) node {$5$};
			\draw (2.5,.5-1-2) node {$4 $};
			\draw (.5,.5-1) node {$2$};
			\draw (1.5,.5-1) node {$0$};
			\draw (2.5,.5-1) node {$0 $};
		\end{large}
	\end{tikzpicture}
	                                                   &  &
	\begin{tikzpicture}[xscale=-1,scale=.5]
		\Yboxdim{1cm}
		\ygreen
		\tyng(0cm,0cm,4,2)
		\ylw
		\tyng(0cm,-3cm,4,4,4)
		\begin{large}
			\draw[line width = .5mm] (-.25,1-1) -- (4.25,1-1);
			\draw[line width = .5mm] (1, -3) -- (1,3.25-1);
			\draw[line width = .5mm] (2, -3) -- (2,3.25-1);
			\draw (.5,.5-1-1) node {$3$};
			\draw (1.5,.5-1-1) node {$1$};
			\draw (2.5,.5-1-1) node {$3 $};
			\draw (3.5,.5-1-1) node {$9 $};
			\draw (.5,.5-1-2) node {$8$};
			\draw (1.5,.5-1-2) node {$2 $};
			\draw (2.5,.5-1-2) node {$3 $};
			\draw (3.5,.5-1-2) node {$4$};
			\draw (.5,.5-1) node {$0$};
			\draw (1.5,.5-1) node {$0 $};
			\draw (2.5,.5-1) node {$2 $};
			\draw (3.5,.5-1) node {$0$};
		\end{large}
	\end{tikzpicture}            &  &
	\begin{tikzpicture}[xscale=-1,scale=.5]
		\Yboxdim{1cm}
		\ygreen
		\tyng(0cm,0cm,1)
		\ylw
		\tyng(0cm,-3cm,1,1,1)
		\begin{large}
			\draw[line width = .5mm] (-.25,1-1) -- (1.25,1-1);
			\draw (.5,.5-1-1) node {$4$};
			\draw (.5,.5-1-2) node {$1$};
			\draw (.5,.5-1) node {$0$};
		\end{large}
	\end{tikzpicture}            &  &
	\begin{tikzpicture}[xscale=-1,scale=.5]
		\Yboxdim{1cm}
		\ylw
		\tyng(0cm,-3cm,3,3,3)
		\draw[line width = .5mm] (-.25,1-1) -- (3.25,1-1);
		\begin{large}
			\draw (.5,.5-1-1) node {$4$};
			\draw (1.5,.5-1-1) node {$1 $};
			\draw (2.5,.5-1-1) node {$1 $};
			\draw (.5,.5-1-2) node {$6$};
			\draw (1.5,.5-1-2) node {$6 $};
			\draw (2.5,.5-1-2) node {$1 $};
			\draw (.5,.5-1) node {$0$};
			\draw (1.5,.5-1) node {$1$};
			\draw (2.5,.5-1) node {$0$};
		\end{large}
	\end{tikzpicture} &  &
	\begin{tikzpicture}[xscale=-1,scale=.5]
		\Yboxdim{1cm}
		\ygreen
		\tyng(0cm,0cm,1)
		\ylw
		\tyng(0cm,-3cm,1,1,1)
		\draw[line width = .5mm] (-.25,1-1) -- (1.25,1-1);
		\begin{large}
			\draw (.5,.5-1) node {$0$};
			\draw ( .5,.5-2) node {$1 $};
			\draw ( .5,.5-3) node {$1 $};
		\end{large}
	\end{tikzpicture}
\end{align*}
\caption{An example of $\psi_{\circ}(S') $.}
	\label{figure:phis_prime}
\end{figure}

\begin{proposition}
	For $|\gamma|>0$, the map $\psi_{\circ}$ defined above is a weight-preserving, sign-reversing involution.
\end{proposition}
\begin{proof}[\bf Proof]
	The proof follows just as in \Cref{section:weights}. Clearly, steps 1 and 2 are inverses. And we know $\varphi$ in step 3 is a weight-preserving, sign-reversing involution. Thus, we need to show that if $S$ falls into cases 1 or 2, then $\psi_{\circ}(S)$ also falls into case 1 or 2. And if $S = (S^1,\dots, S^r)$ falls into case 3, then so does $\psi(S)$. It therefore suffices to show that for $S$ in case 3, $\psi(S) = (T^1,\dots, T^s)$ also falls into case 3.

	Suppose the circle is leading in $S^1$ and $S^r$ cannot join $S^1$. Then this means that $c_1(S^1)< c_\ell(S^r)+ \D(S^r ; S^1)$. By definition of $\psi$, we must always have $c_\ell(S^r) + \D( S^r ; S^1) = c_\ell(T^r) + \D(T^r ; T^1)$. This is because if $\psi$ splits $S^r$ into $(T^{s-1}, T^s)$, then the last column of $S^r$ increases by $\D(T^{s-1};T^s)$ in order to create $c_\ell(T^s)$. In other words, $c_\ell(T^s) = \D( T^{s-1} ; T^s) + c_\ell(S^r)$. Therefore using that $\D(T^{r-1} ; T^r)+ \D( T^r ; T^1) = \D(S^r ; S^1),$
	\[
		c_\ell(T^s) +\D(T^r;T^1)= c_\ell(S^r) + \D(T^{s-1};T^s) + \D(T^r;T^1) =  c_\ell(S^r) + \D(S^r;S^1).
	\]
	If instead $\psi$ joins $S^{r -1}, S^r$ to make $T^s$, then this argument also works if we interchange $T^s, T^{s-1}$ by $S^r, S^{r-1}$.

	There is a problem with this argument when $S = (S^1)$ consists of a single part, since we would not have the condition that $S^\ell$ cannot join $S^1$. Suppose the circle is leading in $S$ and $\psi(S) = (T^1,T^2)$. We have to make sure that $T^2$ does not join with $T^1$; otherwise, we would have something from case 3 landing in case 2. This is only true if $\gamma$ is nonempty:

	If there is a nonempty label $i$ in $T^1$, then $c_\ell(T^2) > c_\ell(S^1) $, since at least $i$ new cells were added to each column to form $T^2$. In particular $c_\ell(T^2) >0$, while (from the definition of $\SC^\gamma(\eta,\mu)$), $c_1(S^1) = c_1(T^1) = 0$. Therefore, $T^2$ cannot join $T^1$.  On the other hand, if there is a nonzero label in $T^2$, then to join $T^2$ and $T^1$, we would need $c_1(T^1) \geq i > 0$. But since $c_1(T^1) = c_1(S^1) = 0$, we have that this cannot be so. Therefore, if $S$ falls into case 3 then so will $\psi_{\circ}(S)$.
\end{proof}

Note that all the fixed points must be in $\SC^\gamma(\eta,\eta)$ for some $\eta$. Let $\bSM_{ \eta}^\gamma$ be the set of fixed points of $\psi_{\circ}$ of type $\eta$. Then we have shown the following.
\begin{proposition}
	For any $n$, $k$, and nonempty partition $\gamma$, 
	\[ \widehat{\nabla}_\ast^k \, \widehat{\Delta}_{m_{\gamma}} \omega(p_n) = \sum_{\mu \vdash n} \sum_{S \in \bSM^\gamma_\eta} \rho(S) \otimes e_\eta. \]
\end{proposition}
The set $\bSM^\gamma_\eta$ is given as follows. An element $S = (S^1,\dots, S^r) \in \bSM^\gamma_\eta$ is fixed by $\psi^{\circ}$ if
\begin{enumerate}
	\item  For all $i$, $S^i$ cannot split, meaning there are no bars. This means that for each $i$, $c_1(S^i) = c_\ell(S^i)$ and we can therefore denote this quantity with no subscripts: $c(S^i) = c_1(S^i) = c_\ell(S^i)$.
	\item   For all $i< \ell$, $S^i$ cannot join $S^{i+1}$, and $S^\ell$ cannot join $S^1$.
	\item  The circle is leading in $S^1$.
\end{enumerate}
So each $S^i$ consists of a single part from $\eta$, $c(S^{i+1}) < c (S^i) +\D(S^i;S^{i+1})$, and also $c (S^1) < c(S^\ell) +\D(S^\ell;S^1)$ .
For example, suppose $\eta = (4,3,1,1)$ and $\gamma = (3,2,2,1,1,1)$. Then to create an element $S = (S^1,S^2, S^3, S^4)  \in \bSM^\gamma_\eta$, first pick a rearrangement of $\eta$, say $(3,1,4,1)$, and select a cell to circle in the first part. These will be the base rows of the tableaux.
\begin{align*}
	\begin{tikzpicture}[xscale=-1,scale=.5]
		\Yboxdim{1cm}
		\ylw
		\tyng(0cm,-1cm,3 )
		\draw[line width = .5mm] (-.25,1-1) -- (3.25,1-1);
		\begin{large}
			\draw ( 1.5,-.5) node {{\Circle}};
		\end{large}
	\end{tikzpicture}
	 &  &
	\begin{tikzpicture}[xscale=-1,scale=.5]
		\Yboxdim{1cm}
		\ylw
		\tyng(0cm,-1cm,1)
		\draw[line width = .5mm] (-.25,1-1) -- (1.25,1-1);
	\end{tikzpicture}
	 &  &
	\begin{tikzpicture}[xscale=-1,scale=.5]
		\Yboxdim{1cm}
		\ylw
		\tyng(0cm,-1cm,4)
		\draw[line width = .5mm] (-.25,1-1) -- (4.25,1-1);
	\end{tikzpicture}
	 &  &
	\begin{tikzpicture}[xscale=-1,scale=.5]
		\Yboxdim{1cm}
		\ylw
		\tyng(0cm,-1cm,1)
		\draw[line width = .5mm] (-.25,1-1) -- (1.25,1-1);
	\end{tikzpicture}
\end{align*}
\noindent Choose
a rearrangement of $\gamma, 0^{|\mu| - \ell(\gamma)}$, say $(1,0,3,1,0,2,1,0,2)$, to place in the base rows.
\begin{align*}
	\begin{tikzpicture}[ scale=.5]
		\Yboxdim{1cm}
		\ylw
		\tyng(0cm,-1cm,3)
		\draw[line width = .5mm] (-.25,1-1) -- (3.25,1-1);
		\begin{large}
			\draw ( 1.5,  .5) node {{\Circle}};
			\draw (.5,-.5) node { $1$};
			\draw (1.5,-.5) node { $0$};
			\draw (2.5,-.5) node { $3$};
		\end{large}
	\end{tikzpicture}
	 &  &
	\begin{tikzpicture}[ scale=.5]
		\Yboxdim{1cm}
		\ylw
		\tyng(0cm,-1cm,1)
		\draw[line width = .5mm] (-.25,1-1) -- (1.25,1-1);
		\begin{large}
			\draw (.5,-.5) node {$1$};
		\end{large}
	\end{tikzpicture}
	 &  &
	\begin{tikzpicture}[ scale=.5]
		\Yboxdim{1cm}
		\ylw
		\tyng(0cm,-1cm,4)
		\draw[line width = .5mm] (-.25,1-1) -- (4.25,1-1);
		\begin{large}
			\draw (.5,-.5) node {$0$};
			\draw (1.5,-.5) node {$2$};
			\draw (2.5,-.5) node {$1$};
			\draw (3.5,-.5) node {$0$};
		\end{large}
	\end{tikzpicture}
	 &  &
	\begin{tikzpicture}[  scale=.5]
		\Yboxdim{1cm}
		\ylw
		\tyng(0cm,-1cm,1)
		\draw[line width = .5mm] (-.25,1-1) -- (1.25,1-1);
		\begin{large}
			\draw (.5,-.5) node {$2$};
		\end{large}
	\end{tikzpicture}
\end{align*}
\noindent Now place two more letters under each of the cells in the bottom row.
\begin{align*}
	\begin{tikzpicture}[ scale=.5]
		\Yboxdim{1cm}
		\ylw
		\tyng(0cm,-3cm,3,3,3)
		\draw[line width = .5mm] (-.25,1-1) -- (3.25,1-1);
		\begin{large}
			\draw ( 1.5,  .5) node {{\Circle}};
			\draw (.5,-.5) node { $1$};
			\draw (1.5,-.5) node { $0$};
			\draw (2.5,-.5) node { $3$};
			\draw (.5,-1.5) node { $1$};
			\draw (1.5,-1.5) node { $1$};
			\draw (2.5,-1.5) node { $3$};
			\draw (.5,-2.5) node { $1$};
			\draw (1.5,-2.5) node { $2$};
			\draw (2.5,-2.5) node { $3$};
		\end{large}
	\end{tikzpicture}
	 &  &
	\begin{tikzpicture}[ scale=.5]
		\Yboxdim{1cm}
		\ylw
		\tyng(0cm,-3cm,1,1,1)
		\draw[line width = .5mm] (-.25,1-1) -- (1.25,1-1);
		\begin{large}
			\draw (.5,-.5) node {$1$};
			\draw (.5,-1.5) node {$1$};
			\draw (.5,-2.5) node {$1$};
		\end{large}
	\end{tikzpicture}
	 &  &
	\begin{tikzpicture}[ scale=.5]
		\Yboxdim{1cm}
		\ylw
		\tyng(0cm,-3cm,4,4,4)
		\draw[line width = .5mm] (-.25,1-1) -- (4.25,1-1);
		\begin{large}
			\draw (.5,-.5) node {$0$};
			\draw (1.5,-.5) node {$2$};
			\draw (2.5,-.5) node {$1$};
			\draw (3.5,-.5) node {$0$};
			\draw (.5,-1.5) node {$3$};
			\draw (1.5,-1.5) node {$2$};
			\draw (2.5,-1.5) node {$1$};
			\draw (3.5,-1.5) node {$2$};
			\draw (.5,-2.5) node {$2$};
			\draw (1.5,-2.5) node {$2$};
			\draw (2.5,-2.5) node {$1$};
			\draw (3.5,-2.5) node {$3$};
		\end{large}
	\end{tikzpicture}
	 &  &
	\begin{tikzpicture}[  scale=.5]
		\Yboxdim{1cm}
		\ylw
		\tyng(0cm,-3cm,1,1,1)
		\draw[line width = .5mm] (-.25,1-1) -- (1.25,1-1);
		\begin{large}
			\draw (.5,-.5) node {$2$};
			\draw (.5,-1.5) node {$1$};
			\draw (.5,-2.5) node {$1$};
		\end{large}
	\end{tikzpicture}
\end{align*}
\noindent To construct $(S^1,\dots, S^4)$ we now choose a value for $c(S^i)$ so that $c(S^{i+1})<c(S^i)+D(S^i;S^{i+1})$ and $c(S^1) < c(S^4)+ D(S^4;S^1)$. We require $c(S^i) = 0$ for at least one $i$.
\begin{align*}
	\begin{tikzpicture}[ scale=.5]
		\Yboxdim{1cm}
		\ygreen
		\tyng(0cm,0cm,3)
		\ylw
		\tyng(0cm,-3cm,3,3,3)
		\draw[line width = .5mm] (-.25,1-1) -- (3.25,1-1);
		\begin{large}
			\draw ( 1.5,  1.5) node {{\Circle}};
			\draw (.5,-.5) node { $1$};
			\draw (1.5,-.5) node { $0$};
			\draw (2.5,-.5) node { $3$};
			\draw (.5,-1.5) node { $1$};
			\draw (1.5,-1.5) node { $1$};
			\draw (2.5,-1.5) node { $3$};
			\draw (.5,-2.5) node { $1$};
			\draw (1.5,-2.5) node { $2$};
			\draw (2.5,-2.5) node { $3$};
		\end{large}
	\end{tikzpicture}
	 &  &
	\begin{tikzpicture}[ scale=.5]
		\Yboxdim{1cm}
		\ylw
		\tyng(0cm,-3cm,1,1,1)
		\draw[line width = .5mm] (-.25,1-1) -- (1.25,1-1);
		\begin{large}
			\draw (.5,-.5) node {$1$};
			\draw (.5,-1.5) node {$1$};
			\draw (.5,-2.5) node {$1$};
		\end{large}
	\end{tikzpicture}
	 &  &
	\begin{tikzpicture}[ scale=.5]
		\Yboxdim{1cm}
		\ygreen
		\tyng(0cm,0cm,4,4)
		\ylw
		\tyng(0cm,-3cm,4,4,4)
		\draw[line width = .5mm] (-.25,1-1) -- (4.25,1-1);
		\begin{large}
			\draw (.5,-.5) node {$0$};
			\draw (1.5,-.5) node {$2$};
			\draw (2.5,-.5) node {$1$};
			\draw (3.5,-.5) node {$0$};
			\draw (.5,-1.5) node {$3$};
			\draw (1.5,-1.5) node {$2$};
			\draw (2.5,-1.5) node {$1$};
			\draw (3.5,-1.5) node {$2$};
			\draw (.5,-2.5) node {$2$};
			\draw (1.5,-2.5) node {$2$};
			\draw (2.5,-2.5) node {$1$};
			\draw (3.5,-2.5) node {$3$};
		\end{large}
	\end{tikzpicture}
	 &  &
	\begin{tikzpicture}[  scale=.5]
		\Yboxdim{1cm}
		\ylw
		\tyng(0cm,-3cm,1,1,1)
		\draw[line width = .5mm] (-.25,1-1) -- (1.25,1-1);
		\begin{large}
			\draw (.5,-.5) node {$2$};
			\draw (.5,-1.5) node {$1$};
			\draw (.5,-2.5) node {$1$};
		\end{large}
	\end{tikzpicture}
\end{align*}
\noindent
In this case, we have the following values. To better illustrate the relative statistics, we have written $D(S^i;S^{i+1}) = a+b$ where $a$ is the sum of the labels in the $0^{\ith}$ row and $b$ is the contribution from the ascents in the lower rows.
\begin{align*}
	D(S^1;S^2) = 4+ 4 &  & D(S^2;S^3) =1 + 2 &  & D(S^3;S^4) = 3+2 &  & D(S^4;S^1) = 2+0 \\
	c(S^1) = 1        &  & c(S^2) = 0        &  & c(S^3) = 2       &  & c(S^4) = 0
\end{align*}
The weight of this sequence of labeled, selected column-composition tableaux is given by
\[
	\rho(S) =  q^{24}  \, x_{11}^5 \, x_{12}^2 \, x_{13}^2 \, x_{21}^4 \, x_{22}^3 \, x_{23}^2.
\]

\begin{definition} \label{retPoly}
	Recall that a parallelogram polyomino $(P,Q) $ consists of two paths that only touch at the beginning and end. If $r$ is the length of the first vertical segment of $P$, then there must exist some $i \geq r$ for which the distance between $P$ and $Q$ along the line $y=i$ is at most $1$ unit. The minimal such $i$ is called the return of $P$ and is denoted by $\ret(P,Q)$.
	Let $\eta \vdash n$ and let $\gamma$ be any nonempty partition. Denote by $\RP_{k,\eta}^\gamma$ the set of $\rho$-compatible polyominoes with a marked return, given by a quadruple $(P,Q,w,i)$, where
	\begin{enumerate}
		\item $w = (w_1,\dots, w_n) \in  (\mathbb{N}^{k+1})^n$, with $w_{10}\cdots w_{n0} \in R(\gamma,0^{n-\ell(\gamma)})$ and $w_{ij}>0$ for $j>0$,
		\item $(P,Q,w)$ is a $\rho$-compatible parallelogram polyomino of type $\eta$ with $P$ ending in two consecutive east steps, and
		\item $1 \leq i \leq \ret(P,Q)$.
	\end{enumerate}
\end{definition}

\begin{definition}
	\label{DeltaPowerBijection}
	Define the map $\varphi_{\circ}: \bSM^{\gamma}_\eta \rightarrow  \RP_{k,\eta}^\gamma$ using $\varphi$ from \Cref{section:weights} as follows. For $S=(S^1,\dots, S^\ell) \in \bSM^{\gamma}_\eta$,
	\begin{enumerate}
		\item if $c (S^1) = 0$, then set $\varphi_{\circ} (S) = \varphi(S)$, keep the circle from $S^1$ inscribed in its corresponding north step of $\varphi(S)$, as is associated by the map $\varphi$ from \Cref{proposition:phi};
		\item if $c(S^1) >0$, then there is a largest $j$ such that $c(S^j) = 0$, so set \[ \varphi_\circ (S) = \varphi(S^j,\dots, S^\ell, S^1, \dots, S^{j-1}),\] again keeping the circle from $S^1$ in its corresponding north step.
	\end{enumerate}
\end{definition}
For instance, to find the image under $\varphi_\circ$ for the last example, we first rearrange the parts to give
\begin{align*}
	                                                   &  &
	\begin{tikzpicture}[  scale=.5]
		\Yboxdim{1cm}
		\ylw
		\tyng(0cm,-3cm,1,1,1)
		\draw[line width = .5mm] (-.25,1-1) -- (1.25,1-1);
		\begin{large}
			\draw (.5,-.5) node {$2$};
			\draw (.5,-1.5) node {$1$};
			\draw (.5,-2.5) node {$1$};
		\end{large}
	\end{tikzpicture} &  &
	\begin{tikzpicture}[ scale=.5]
		\Yboxdim{1cm}
		\ygreen
		\tyng(0cm,0cm,3)
		\ylw
		\tyng(0cm,-3cm,3,3,3)
		\draw[line width = .5mm] (-.25,1-1) -- (3.25,1-1);
		\begin{large}
			\draw ( 1.5,  1.5) node {{\Circle}};
			\draw (.5,-.5) node { $1$};
			\draw (1.5,-.5) node { $0$};
			\draw (2.5,-.5) node { $3$};
			\draw (.5,-1.5) node { $1$};
			\draw (1.5,-1.5) node { $1$};
			\draw (2.5,-1.5) node { $3$};
			\draw (.5,-2.5) node { $1$};
			\draw (1.5,-2.5) node { $2$};
			\draw (2.5,-2.5) node { $3$};
		\end{large}
	\end{tikzpicture}
	                                                   &  &
	\begin{tikzpicture}[ scale=.5]
		\Yboxdim{1cm}
		\ylw
		\tyng(0cm,-3cm,1,1,1)
		\draw[line width = .5mm] (-.25,1-1) -- (1.25,1-1);
		\begin{large}
			\draw (.5,-.5) node {$1$};
			\draw (.5,-1.5) node {$1$};
			\draw (.5,-2.5) node {$1$};
		\end{large}
	\end{tikzpicture}
	                                                   &  &
	\begin{tikzpicture}[ scale=.5]
		\Yboxdim{1cm}
		\ygreen
		\tyng(0cm,0cm,4,4)
		\ylw
		\tyng(0cm,-3cm,4,4,4)
		\draw[line width = .5mm] (-.25,1-1) -- (4.25,1-1);
		\begin{large}
			\draw (.5,-.5) node {$0$};
			\draw (1.5,-.5) node {$2$};
			\draw (2.5,-.5) node {$1$};
			\draw (3.5,-.5) node {$0$};
			\draw (.5,-1.5) node {$3$};
			\draw (1.5,-1.5) node {$2$};
			\draw (2.5,-1.5) node {$1$};
			\draw (3.5,-1.5) node {$2$};
			\draw (.5,-2.5) node {$2$};
			\draw (1.5,-2.5) node {$2$};
			\draw (2.5,-2.5) node {$1$};
			\draw (3.5,-2.5) node {$3$};
		\end{large}
	\end{tikzpicture}
\end{align*}

Applying $\varphi$ then gives the selected polyomino of \Cref{figure:selected_polyomino}.
\begin{figure}[!ht]
	\begin{center}
		\includegraphics{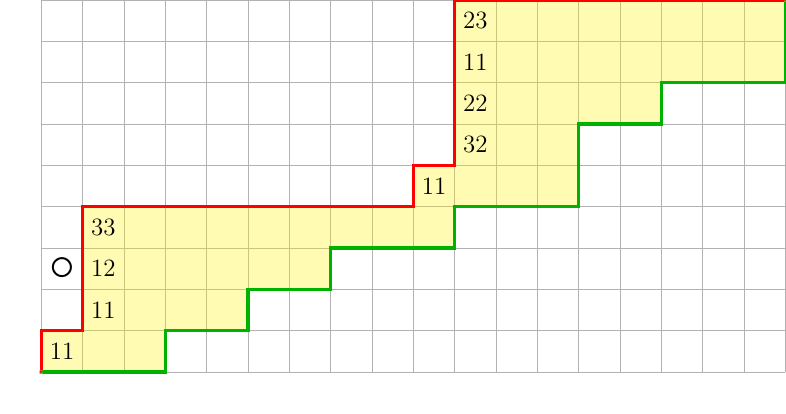}
	\end{center}
	\caption{A $\rho$-compatible labeled polyomino with a marked return in $\RP^{(3,2,2,1,1,1)}_{2,(4,3,1,1)}$ and with area $24$. The paths return when $y=4$.}
	\label{figure:selected_polyomino}
\end{figure}

Let $\RP^\gamma_k = \bigcup_\eta \RP^\gamma_{k,\eta}$, and let $\eta(L) = \eta$ if $L \in \RP^\gamma_{k,\eta}$.

\begin{proposition}
	\label{prop:power}
	The map $\varphi_{\circ}$ is a weight-preserving bijection. In particular, we have
	\[ \widehat{\nabla}_\ast^k \, \widehat{\Delta}_{m_\gamma} \omega(p_n) = \sum_{L \in \RP^\gamma_{k}} q^{\area(L)}  \XX^L \otimes e_{\eta(L)} \]
\end{proposition}

\begin{proof}[\bf Proof]
	The bijection $\varphi$ gives us an injection into a subset labeled parallelogram polyominoes with a marked return. The only condition we need to ensure is that if $\varphi_{\circ}(S)$ gives the polyomino $(P,Q)$, then $P$ ends with two east steps.
	If $S = (S^1,\dots, S^\ell)$ is sent to $\varphi_\circ (S) = \varphi(S^j,\dots, S^\ell, S^1, \dots, S^{j-1})$ with underlying polyomino $(P,Q)$, then we know that $c(S^j) = 0$ and $c(S^j) < c(S^{j-1}) + D(S^{j-1};S^j)$ (since no two consecutive $S^i$ can join). Thus $c(S^{j-1}) + D(S^{j-1};S^j) > c(S^j) = 0$, meaning one of these quantities is nonzero. If $D(S^{j-1};S^j) > 0$, then the last north-segment of $P$ must lie at least two units West of the last point of the path. If $c(S^{j-1}) >0$, then the number of cells between the last vertical segment of $P$ and the bottom path is at least $2$. This gives that in either case, $P$ must end with two east steps.
\end{proof}

We now apply the map $\iota$ of \Cref{def:iota} which sends $\rho$-compatible polyominoes with a marked return to multi-labeled $(\gamma,k^n)$-Dyck paths with a marked return, defined as follows.

\begin{definition}
	Let $\gamma$ be any nonempty partition.
	A $(\gamma,k^n)$-Dyck path is a polyomino $(P,Q)$ where $Q$ is a $(\gamma,k^n)$-staircase, meaning that
	$Q = EE^{a_1} N  \cdots E^{a_n} N$ with \[ (a_1-k,\dots, a_n-k) \in R(\gamma,0^{n-\ell(\gamma)}). \]
	Denote by $\RLD^\gamma_{k^n}$ the set of multi-labeled $(\gamma,k^n)$-Dyck paths with a marked return whose labels are given by words in $\mathbb{N}_+$ of length $k$. Recall that if $L \in \RLD^\gamma_{k^n}$ has labels given by $w(L)= (w_1,\dots, w_n)$, and $i$ is an $r$-descent, then the $i^{\ith}$  north step of the top path must be followed by at least $r$ east steps.
\end{definition}
The map $\iota$ can be applied to the set $\RP^{\gamma}_{k^n}$ as we did before, shifting both the path $P$ and $Q$ by the number of non-ascents in the labels.
\begin{proposition}
	For any $k$ and nonempty partition $\gamma$, we have
	\[ \widehat{\nabla}_\ast^k \,\widehat{\Delta}_{m_\gamma} \omega(p_n) = \sum_{L \in \RLD^\gamma_{k^n}} q^{\area(L)} \XX^L \otimes e_{\eta(L)} \]
\end{proposition}

For instance, the previous example of a marked return polyomino maps to the multi-labeled $(\gamma,k^n)$-Dyck path with a marked return in \Cref{figure:gammaknDyckpath}.

\begin{figure}[!ht]
	\begin{center}
		\includegraphics{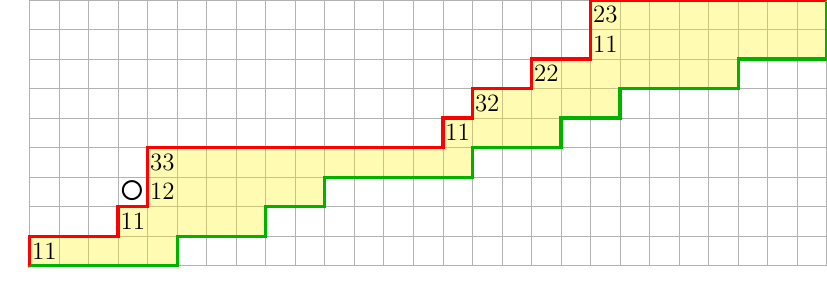}
	\end{center}
	\caption{A multi-labeled $(322111,2^9)$-Dyck path with a marked return.}
	\label{figure:gammaknDyckpath}
\end{figure}

\subsection{A Schur function expansion for applications to power sums}
\label{sub_SchurPower}

We can follow the same procedure outlined in the prior section, but instead use labels giving the Schur expansion; this corresponds to taking labels which form lattice words. We are immediately led to consider the following subset of multi-labeled $(\gamma, k^n)$-Dyck paths.

\begin{definition}
	Let $\gamma$ be a nonempty partition.
	The collection of lattice multi-labeled $(\gamma,k^n)$-Dyck paths with a marked return, denoted by $\RLLD^\gamma_{k^n}$, is the subset of elements $(P,Q,w,m) \in \RLD^{\gamma}_{k^n}$ such that
	\begin{enumerate}
		\item The labels $w_1,\dots, w_n$ are tuples of length $k$.
		\item The words $w_{\star i} = w_{1i}\cdots w_{ni}$ are lattice words. We denote the partition associated to this lattice word by $\lambda^i(w)$.
		\item The top path ends with two east steps.
		\item There is a marked row $m$ before the first return of the paths $(P,Q)$.
	\end{enumerate}
\end{definition}

We have the following.

\begin{proposition}
	For any $k$ and nonempty partition $\gamma$, \[ \widehat{\nabla}_\ast^k \widehat{\Delta}_{m_\gamma} \omega(p_n) = \sum_{L \in \RLLD^\gamma_{k^n}} q^{\area(L)} e_{\eta(L)} \otimes s_{\lambda^1(w(L))} \otimes \cdots \otimes s_{\lambda^k(w(L))}. \]
\end{proposition}

\subsection{Square Paths}
A combinatorial description of the effect of $\nabla$ on a power sum symmetric function is given by the Square Paths Conjecture of Loehr and Warrington \cite{Loehr-Warrington-square-2007}, proved by Emily Sergel in her thesis \cite{Leven-2016}. The conjecture states that we can write $\nabla \omega(p_n)$ as a sum over labeled square paths. We now define these structures and show that when $k=0$ and $\gamma = 1^n$ our results coincide to the Square Paths Theorem when $t=1$.

A square path $\pi \in \SP^E_n$ is a lattice path consisting of north steps and east steps from $(0,0)$ to $(n,n)$ that ends in an east step. The set of labeled square paths $\LSP^E_n$ is generated by placing positive integers along the  north steps of the square path so that the columns are increasing when read from bottom to top. The monomial of a labeled square path $L \in \LSP$ is given by $\X^L = x_1^{a_1} x_2^{a_2} \cdots$ where $a_i$ is the number of occurrences of the label $i$ in $L$. For example, the square path in \Cref{fig:square_paths} (left) has monomial weight $x_1 x_2^3 x_3 x_4^2 x_5$.

\begin{figure}[ht]
	\begin{tikzpicture}[scale=.6]
		\draw[step=1.0, gray!60, thin] (0,0) grid (8,8);
		\begin{scope}
			\clip (0,0) rectangle (8,8);
			\draw[gray!60, thin] (2,0) -- (8,6);
		\end{scope}
		\draw[blue!60, line width=2pt] (0,0) -- (0,1) -- (3,1) -- (3,4) --(6,4) --(6,8) --(8,8);
		\draw (.5,.5) node {$2$};
		\draw (3.5,1.5) node {$1$};
		\draw (3.5,2.5) node {$2$};
		\draw (3.5,3.5) node {$4$};
		\draw (6.5,7.5) node {$5$};
		\draw (6.5,4.5) node {$2$};
		\draw (6.5,5.5) node {$3$};
		\draw (6.5,6.5) node {$4$};
	\end{tikzpicture}
	\begin{tikzpicture}
		\draw[draw = none, use as bounding box] (0,0) rectangle (2,2);
	\end{tikzpicture}
	\begin{tikzpicture}[scale=.6]
		\draw[step=1.0, gray!60, thin] (0,0) grid (8,8);
		\begin{scope}
			\clip (0,0) rectangle (8,8);
			\draw[gray!60, thin] (0,0) -- (8,8);
		\end{scope}
		\draw[blue!60, line width=2pt] (0,0) -- (0,4) -- (2,4) -- (2,5) --(5,5) -- (5,8) --(8,8);
		\draw (2.5,4.5) node {$2$};
		\draw (5.5,5.5) node {$1$};
		\draw (5.5,6.5) node {$2$};
		\draw (5.5,7.5) node {$4$};
		\draw (.5,3.5) node {$5$};
		\draw (.5,.5) node {$2$};
		\draw (.5,1.5) node {$3$};
		\draw (.5,2.5) node {$4$};
		\draw (1.5,1.5) node {$\Circle$};
	\end{tikzpicture}
	\caption{A labelled square path (left) and its marked circular rearrangement (right)}
	\label{fig:square_paths}
\end{figure}
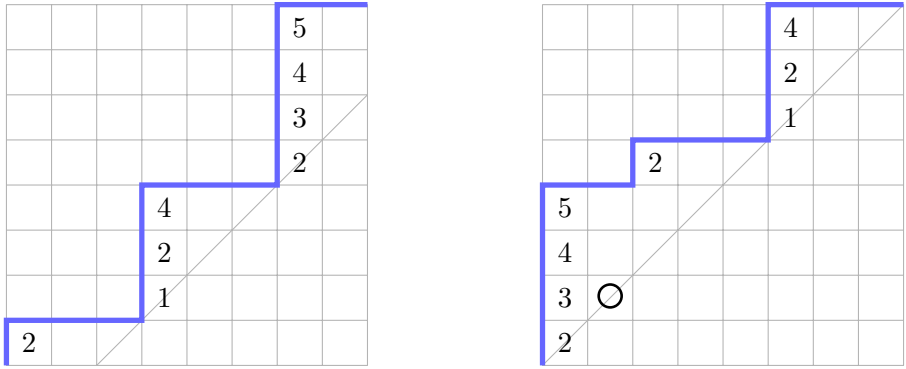

If the  north step on the lowest diagonal starts on diagonal $y=x-d$, then a north step on diagonal $y = x +a$ will contribute $a + d$ units of area. The polyomino in \Cref{figure:gammaknDyckpath} would then have area $11$. The number of diagonal inversions is calculated in a similar way to Dyck paths. Then the square paths theorem can be stated as \[ \nabla \omega(p_n) = \sum_{L \in \LSP^E_{n}} q^{\area(L)} t^{\dinv(L)} \X^L. \]

If we circularly rearrange the north steps and east steps of the square path above, so that the path begins with the right-most, lowest vertical segment, and we circle a cell on the main diagonal to mark the last column of the original square path, we get a marked Dyck path, as in \Cref{fig:square_paths} (right). This corresponds to marking a row before the Dyck path returns to the main diagonal. The area is simply the area of the Dyck path. Recall that \[ \left. \nabla \omega(p_n) \right\rvert_{t=1} = \widehat{\Delta}_{m_{1^n}} \omega(p_n) = \sum_{\mu \vdash n} e_\mu \sum_{L \in \RP_{\lambda,1^n}} t^{\area(L)}. \]

The first thing to note is that if $(P,Q) \in \RP_{\lambda,1^n}$, then the bottom $Q$ is completely determined: the path must travel from $(1,0)$ to $(n+1,n)$ with only north steps and east steps; and every east step must be followed by a  north step (since $\lambda = 1^n$). This means that $Q$ is the path which begins with an east step, then alternates ``east, north, east, north'', and so on. On the other hand $P$ cannot touch $Q$, so it must remain weakly above the line $y=x$ until the last step. Therefore, $P$ is a Dyck path. The return of the polyomino is equal to the return of $P$ as a Dyck path, i.e. the first time $P$ returns to the diagonal $y=x$. This equates the two interpretations when we set $t=1$.

Another application of our formulas is a proof of the Delta Square Conjecture \cite{DAdderio-Iraci-VandenWyngaerd-DeltaSquare-2019} when $t=1$. Here, there is a formula for $\Delta_{h_a} \Delta_{e_b} \omega(p_n)$ in terms of (partially) labeled square paths; our methods establish an interpretation for this symmetric function that agrees with the conjecture when $t=1$.

\section{Further directions}

\subsection{Representation theory}

In the special case when $k=2$, the expression $\nabla_\ast e_n$ appears to have positive coefficients in the Schur basis. The following conjecture has been checked by computer for $n \leq 9$.

\begin{conjecture}
	\label{conj:schur}
	The expression $\nabla_\ast e_n$ is $(s \otimes s)$-positive.
\end{conjecture}

If \Cref{conj:schur} holds, then it suggests that the expression $\nabla_\ast e_n$ has some representation-theoretical meaning; it would be interesting to find such an interpretation. The $(s \otimes s)$-expansion for $\nabla_\ast e_4$ is given in \Cref{table:nabla_e4}.

\begin{table}[ht]
    \resizebox{\textwidth}{!}{%
	\centering
	\renewcommand{\arraystretch}{1.5}
    \begin{tabular}{|c|c|c|}
		\hline
        $\mathbf{\mu}$ & $\mathbf{\nu}$ & $\textbf{Coefficient of } \mathbf{s_\mu \otimes s_\nu}$ \\
        \hline
        \hline
        $(1, 1, 1, 1)$ & $(1, 1, 1, 1)$ & $\makecell{q^{6} + q^{5} t + q^{4} t^{2} + q^{3} t^{3} + q^{2} t^{4} + q t^{5} + t^{6} + q^{4} t + q^{3} t^{2} + q^{2} t^{3} + q t^{4} + q^{3} t + q^{2} t^{2} + q t^{3}}$ \\
        \hline
        $(2, 1, 1)$ & $(1, 1, 1, 1)$ & $\makecell{q^{5} + q^{4} t + q^{3} t^{2} + q^{2} t^{3} + q t^{4} + t^{5} + q^{4} + 2 q^{3} t + 2 q^{2} t^{2} + 2 q t^{3} + t^{4} + q^{3} + 2 q^{2} t + 2 q t^{2} + t^{3} + q t}$ \\
        \hline
        $(3, 1)$ & $(1, 1, 1, 1)$ & $\makecell{q^{3} + q^{2} t + q t^{2} + t^{3} + q^{2} + q t + t^{2} + q + t}$ \\
        \hline
        $(4)$ & $(1, 1, 1, 1)$ & $\makecell{1}$ \\
        \hline
        $(2, 2)$ & $(1, 1, 1, 1)$ & $\makecell{q^{4} + q^{3} t + q^{2} t^{2} + q t^{3} + t^{4} + q^{2} t + q t^{2} + q^{2} + q t + t^{2}}$ \\
        \hline
        $(1, 1, 1, 1)$ & $(2, 1, 1)$ & $\makecell{q^{5} + q^{4} t + q^{3} t^{2} + q^{2} t^{3} + q t^{4} + t^{5} + q^{4} + 2 q^{3} t + 2 q^{2} t^{2} + 2 q t^{3} + t^{4} + q^{3} + 2 q^{2} t + 2 q t^{2} + t^{3} + q t}$ \\
        \hline
        $(2, 1, 1)$ & $(2, 1, 1)$ & $\makecell{q^{4} + q^{3} t + q^{2} t^{2} + q t^{3} + t^{4} + 2 q^{3} + 3 q^{2} t + 3 q t^{2} + 2 t^{3} + 2 q^{2} + 3 q t + 2 t^{2} + q + t}$ \\
        \hline
        $(3, 1)$ & $(2, 1, 1)$ & $\makecell{q^{2} + q t + t^{2} + q + t + 1}$ \\
        \hline
        $(2, 2)$ & $(2, 1, 1)$ & $\makecell{q^{3} + q^{2} t + q t^{2} + t^{3} + q^{2} + 2 q t + t^{2} + q + t}$ \\
        \hline
        $(1, 1, 1, 1)$ & $(2, 2)$ & $\makecell{q^{4} + q^{3} t + q^{2} t^{2} + q t^{3} + t^{4} + q^{2} t + q t^{2} + q^{2} + q t + t^{2}}$ \\
        \hline
        $(2, 1, 1)$ & $(2, 2)$ & $\makecell{q^{3} + q^{2} t + q t^{2} + t^{3} + q^{2} + 2 q t + t^{2} + q + t}$ \\
        \hline
        $(3, 1)$ & $(2, 2)$ & $\makecell{q + t}$ \\
        \hline
        $(2, 2)$ & $(2, 2)$ & $\makecell{q^{2} + q t + t^{2} + 1}$ \\
        \hline
        $(1, 1, 1, 1)$ & $(3, 1)$ & $\makecell{q^{3} + q^{2} t + q t^{2} + t^{3} + q^{2} + q t + t^{2} + q + t}$ \\
        \hline
        $(2, 1, 1)$ & $(3, 1)$ & $\makecell{q^{2} + q t + t^{2} + q + t + 1}$ \\
        \hline
        $(2, 2)$ & $(3, 1)$ & $\makecell{q + t}$ \\
        \hline
        $(1, 1, 1, 1)$ & $(4)$ & $\makecell{1}$ \\
        \hline
    \end{tabular}
	}
    \caption{The $(s \otimes s)$-expansion of $\nabla_\ast e_4$.}
	\label{table:nabla_e4}
\end{table}

It is worth mentioning that $s^{\otimes k}$-positivity fails when $k > 2$. Indeed, the coefficient of $s_{42} \otimes s_{42} \otimes s_{33}$ in $\nabla_\ast^2 e_6$ fails to be $(q,t)$-positive, as its degree $0$ coefficient is $-1$. Once we specialize $t=1$ we recover $q$-positivity, in accordance with \Cref{thm:schur-expansion}.

\subsection{Monomial expansion}

Although we give explicit combinatorial models in the $t=1$ case for several general applications of $\nabla_\ast$ operators, we have no such
combinatorial descriptions for the general setup of $(q,t)$ parameters. Indeed, computer experiments suggest the following conjecture, checked for $k, n \leq 6$.

\begin{conjecture}
	The expression $\nabla_\ast^{k-1} e_n$ is $m^{\otimes k}$-positive.
\end{conjecture}

This suggests the existance of a $t$-statistic on multilabeled $kn \times n$-Dyck paths that refines \Cref{thm:monomial-expansion} to the following identity.

\begin{problem}
	\label{problem:tstat}
	Find a statistic $\mathsf{tstat} \colon \LD_{k^n} \rightarrow \LD_{k^n}$ such that
	\begin{equation}
		\nabla_\ast^k e_n = \sum_{(\pi,w) \in \LD_{k^n}} q^{\area(\pi)} t^{\mathsf{tstat}(\pi, w)} \XX^w.
	\end{equation}
\end{problem}

It is possible that such descriptions may be found exploiting the work of Carlsson and Mellit (see \cites{Mellit-Poincare, Carlsson-Mellit-Nabla}). Building on an expansion of Mellit for
\begin{align}
	\label{Mellit-Poincare}
	\sum_{n=0}^{\infty} \sum_{\mu \vdash n} \frac {\Ht _\mu(\X_0) \cdots \Ht _\mu(\X_{k})}{w_\mu},
\end{align}
Carlsson and Mellit \cite{Carlsson-Mellit-Nabla} obtain a new proof of the shuffle theorem for $\nabla^k e_n$. In their proof, they specialize $\X_0$ to $M$, and take the coefficient of $s_{1^n}(\X_i)$ for $1 \leq i \leq k$ to reduce \Cref{Mellit-Poincare} to $\sum_{n=0}^{\infty} \nabla^k e_n$.  With this specialization, they show that the infinite series involves nilpotent endomorphisms which allows a further reduction to a finite series of monomials in $\X$, with the correct positive $q,t$-weights.

An interesting byproduct of their method is that the statistic $\text{dinv}$ pops out naturally, without having to know the shuffle conjecture beforehand. Further study may yield a refinement of their argument giving a positive monomial expansion for $\nabla_\ast^k e_n$ and a solution to \Cref{problem:tstat}.

\section*{Acknowledgements}

A. Iraci acknowledges the MIUR Excellence Department Project awarded to the Department of Mathematics, University of Pisa, CUP I57G22000700001.

M. Romero was partially supported by the NSF Mathematical Sciences Postdoctoral Research Fellowship DMS-1902731.

\bibliographystyle{amsalpha}
\bibliography{bibliography}

\end{document}